\documentclass[a4paper]{amsart}

\usepackage[utf8]{inputenc}
\usepackage[british]{babel}
\usepackage{amsfonts,amsmath,amssymb,amsthm,mathtools,esint}
\usepackage[dvipsnames]{xcolor}
\usepackage{tikz}
\usepackage{pgfplots}
\pgfplotsset{compat=1.3}
\usepackage{caption}
\usepackage[justification=centering]{caption}
\usepackage{enumitem}
\usepackage{appendix}

\usepackage[hidelinks]{hyperref}
\usepackage{cleveref}
\crefname{subsection}{subsection}{subsections}
\numberwithin{equation}{section}

\usepackage[backend=bibtex, giveninits=true, isbn=false, doi=false, url=false, maxnames = 10]{biblatex}
\DeclareFieldFormat[article]{title}{{#1}} 
\DeclareFieldFormat[book]{title}{{#1}}
\DeclareFieldFormat[inproceedings]{title}{{#1}}
\DeclareFieldFormat[incollection]{title}{{#1}}
\DeclareFieldFormat{pages}{{#1}} 
\renewbibmacro{in:}{%
	\ifentrytype{article}
	{}
	{\bibstring{in}%
		\printunit{\intitlepunct}}}
\AtBeginBibliography{\footnotesize}

\def\R{\mathbb{R}}
\def\F{\mathcal{F}}
\def\E{\mathcal{E}}
\def\T{\mathcal{T}}

\def\pw{\mathrm{pw}}

\newcommand\ds{\,\text{d}s}
\newcommand\dt{\,\text{d}t}
\newcommand\dx{\,\text{d}x}

\newtheorem{theorem}{Theorem}[section]

\newtheorem{lemma}[theorem]{Lemma}

\theoremstyle{definition}

\theoremstyle{remark}
\newtheorem{remark}[theorem]{Remark}
\numberwithin{theorem}{section}
\numberwithin{equation}{section}
\numberwithin{table}{section}
\numberwithin{figure}{section}

\begin{document}
	
\title[HHO for biharmonic problem]{A hybrid high-order method for the biharmonic problem}
\author[Y. Liang, N.~T.~Tran]
      {Yizhou Liang \and Ngoc Tien Tran}
\address{Mathematical Institute, University of Oxford, Woodstock Rd, OX2 6GG Oxford, UK}

\email{yizhou.liang@maths.ox.ac.uk}

\address{Institute of Mathematics, University of Augsburg, Universit\"atsstr.~12a, 86159 Augsburg, Germany}

\email{ngoc1.tran@uni-a.de}

\thanks{The research of the first author was partly supported through a Royal Society University Research Fellowship (URF\textbackslash R1\textbackslash 221398, RF\textbackslash ERE\textbackslash 221047). The second author received funding from the European Union's Horizon 2020 research and innovation programme (project RandomMultiScales, grant agreement No.~865751)}



\maketitle

\begin{abstract}
    This paper proposes a new hybrid high-order discretization for the biharmonic problem and the corresponding eigenvalue problem. The discrete ansatz space includes degrees of freedom in $n-2$ dimensional submanifolds (e.g., nodal values in 2D and edge values in 3D), in addition to the typical degrees of freedom in the mesh and on the hyperfaces in the HHO literature. This approach enables the characteristic commuting property of the hybrid high-order methodology in any space dimension. The main results are guaranteed lower eigenvalue bounds of higher order. Furthermore, we derive quasi-best approximation estimates as well as reliable and efficient a~posteriori error estimators under minimal regularity assumptions on the exact solution. The latter motivates an adaptive mesh-refining algorithm that empirically recovers optimal convergence rates for singular solutions.
\end{abstract}

\section{Introduction}
Let $\Omega \subset \mathbb{R}^n$, $n \in \{2,3\}$, be a polyhedral Lipschitz domain with boundary $\partial \Omega$ and unit outer normal vector $\nu_{\partial \Omega}$. 
The main focus of this paper is the approximation of the eigenpairs $(\lambda,u)$ of the biharmonic eigenvalue problem
\begin{equation}\label{def:pde}
    \begin{aligned}
        \Delta^2 u &= \lambda u \quad\text{in }\Omega,\\
        u &= 0 \quad\text{on }\partial\Omega,\\
        \partial_n u &= 0 \quad\text{on }\partial\Omega.
    \end{aligned}
\end{equation}
Here and throughout this paper, the subscripts $n$ and $t$ are associated with differential operators in the normal and tangential directions, respectively. The variational formulation of \eqref{def:pde} seeks $(\lambda,u) \in \mathbb{R}_{\geq 0} \times V$ with $V \coloneqq  H^2_0(\Omega)$ such that
\begin{align}\label{def:pde-var}
    (\nabla^2 u, \nabla^2 v)_\Omega = \lambda(u,v)_\Omega \quad\text{for any } v \in V.
\end{align}
For the source problem, there are numerous finite element methods for the approximation of \eqref{def:pde-var}. Conforming methods include the piecewise quintic Argyris element in 2D \cite{ArgyrisFriedScharpf1968} or $C^1$ conforming elements in a subtriangulation, also known as Hsieh-Clough-Tocher splits in 2D \cite{Ciarlet2002,GuzmanLischkeNeilan2022} and Worsey–Farin splits in 3D \cite{WorseyFarin1987,GuzmanLischkeNeilan2022}. Due to the complicated nature of $C^1$ conforming elements in the implementation, in particular in 3D, many nonconforming methods have been developed. On simplicial meshes, we mention the lowest-order Morley element \cite{Morley1968} as well as its higher-order generalization in 2D \cite{BlumRannacher1990}. 
Finite elements of arbitrary order on general meshes include the virtual elements (VE) of \cite{BrezziMarini2013,ZhaoChenZhang2016,AntoniettiManziniVerani2018} in 2D, the weak Galerkin (WG) methods of \cite{MuWangYe2014,ZhangZhai2015}, the hybrid high-order (HHO) methods of \cite{BonaldiDiPietroGeymonat2018,DongErn2022,DongErn2024}, and the discontinuous Galerkin (dG) methods of \cite{MozolevskiSuli2003,GeorgoulisHouston2009}.
The degrees of freedom of the aforementioned methods are different and lead to varying computational efficiency. In structured meshes, the performance of all methods is fairly comparable; cf.~\cite{DongErn2022} for a comparison between some HHO, WG, dG and the Morley finite element methods.

The focus of this paper is the HHO methodology introduced in \cite{DiPietroErnLemaire2014,DiPietroErn2015}; we refer to the monographs \cite{DiPietoDroniou2020,CicuttinErnPignet2021} for an overview of applications.
The HHO methodology is a class of hybridizable methods that can be defined for arbitrary polytopal meshes and polynomial degree -- a flexibility shared by many nonconforming methods from above. In fact, close relationship between hybridizable methods such as HDG, HHO, and WG exists, cf.~\cite{Cockburn2016,CockburnDiPietroErn2016,CicuttinErnPignet2021} for further discussions. In certain applications, hybridizable methods allow for additional benefits. We mention guaranteed error control in convex minimization problems by duality techniques, where HHO methods provide a cheap post-processing of the dual variable \cite{CarstensenTran2021,Tran2024}.

For eigenvalue problems, upper bounds can be obtained by conforming methods. However, applications, e.g., in safety analysis require guaranteed lower eigenvalue bounds (LEB), which is accessible by hybridizable methods \cite{CarstensenZhaiZhang2020,CarstensenGraessleTran2024,Tran2024x2}. Here, direct lower bounds can be obtained with a fine-tuned stabilization, i.e., the discrete eigenvalue is a lower bound of the exact one. This leads to
higher-order convergence rates and provides access to adaptive mesh-refining algorithm. Thus, the HHO eigensolver of this paper can be considered as a higher-order generalization of \cite{CarstensenPuttkammer2023}.
It turns out that these benefits are related to the design of the reconstruction operator, which satisfies the following commuting property.
Let $V_h$ denote the discrete ansatz space and $I_h : V \to V_h$ is a canonical interpolation of $V$ onto $V_h$.
The Hessian of a continuous function is approximated by the piecewise Hessian of a reconstruction operator $R_h : V_h \to P_{k+2}(\mathcal{T})$, where $k \geq 0$ is the degree of approximation and $\mathcal{T}$ is the underlying mesh, such that the $L^2$ orthogonality
\begin{align}\label{eq:commuting}
    \nabla^2_\pw (v - R_h I_h v) \perp \nabla^2_\pw P_{k+2}(\mathcal{T}) \quad\text{for any } v \in V
\end{align}
holds.
This property is available in 2D for \cite{DongErn2022} and in arbitrary space dimensions for \cite{BonaldiDiPietroGeymonat2018}, although the design of \cite{DongErn2022} appears more sophisticated as it utilizes less (non-hybridizable) degrees of freedom. To achieve \eqref{eq:commuting} in any space dimensions and derive LEB for \eqref{def:pde-var}, we utilize an ansatz similar to \cite{LiWangWangZhang2024}. The point of departure is the local integration by parts formula
\begin{align}\label{eq:iby-de}
    (\nabla^2 v, \nabla^2 p)_{T} = (v, \Delta^2 p)_{T} - (v, \partial_n \Delta p)_{\partial T} + (\partial_n v, \partial_{nn} p)_{\partial T} + (\partial_t v, \partial_{nt} p)_{\partial T}
\end{align}
for any $v \in H^2(T)$ and $p \in P_{k+2}(T)$ in a mesh element $T$
from \cite{DongErn2022}.
For the Poisson equation, lowest-order hybridizable and Crouzeix-Raviart methods are closely related.
Thus, considering the degrees of freedom of Morley finite elements motivates another integration by parts for the term $(\partial_t v, \partial_{nt} p)_{\partial T}$ along the $n-1$ dimensional hyperfaces $F$ of $T$ with outer normal vector $\nu_{\partial F}$. We arrive at
\begin{align}\label{eq:ibp}
    \begin{aligned}
        (\nabla^2 v, \nabla^2 p)_T & = (v, \Delta^2 p)_T - (v, \partial_n\Delta p)_{\partial T} + (\partial_n v, \partial_{nn} p)_{\partial T} \\ &\quad - (v,\Delta_t\partial_n p)_{\partial T} + \sum_{F\in\mathcal{F}(T)} (v,\partial_{tn} p \cdot \nu_{\partial F})_{\partial F},
    \end{aligned}
\end{align}
where $\partial F$ denotes the $n-2$ dimensional boundary of the face $F$.
(In 2D, integrals over vertices are understood as the point evaluation.)
This formula leads to a discrete ansatz space that also involves degrees of freedom to approximate the restriction of $v$ along $\partial F$ similar to \cite{LiWangWangZhang2024} and the virtual elements \cite{BrezziMarini2013,ZhaoChenZhang2016,AntoniettiManziniVerani2018}.
These additional degrees of freedom ensure that the resulting reconstruction operator $R_h$ satisfies \eqref{eq:commuting}. From a computational point of view, the degrees of freedom along the edges lead to a larger stencil compared to \cite{DongErn2022} and, thus, higher computational cost is expected, cf. \Cref{rem:computational-cost} below for further details. This is, however, justified for the task at hand.
In combination with a new stabilization, we derive the first LEB with higher convergence rates for the biharmonic problem in the current literature with the arguments of \cite{CarstensenZhaiZhang2020,Tran2024x2}.
To be precise,
we establish that
\begin{align}\label{ineq:LEB}
	\mathrm{LEB}(j) \coloneqq \min\{1,1/(\alpha +\beta\lambda_h(j))\}\lambda_h(j) \leq\lambda(j)
\end{align}
is a lower bound for the $j$-th exact eigenvalue $\lambda(j)$. Here, $\lambda_h(j)$ is the $j$-th discrete eigenvalue,
$\alpha =\sigma(1/\pi^4 + c_\mathrm{tr}/\pi^2 + c_\mathrm{tr} + c_\mathrm{tr}(2/\pi + n/\pi^2))$ with a parameter $\sigma$ chosen by the user and the constant $c_\mathrm{tr}$ from the continuous trace inequality, and $\beta =  h^4/\pi^4$. Notably, $\alpha$ and $\beta$ are explicit and independent of the polynomial degree $k$.
In particular, if $\alpha < 1$ and the mesh size $h$ is sufficiently small, then $\lambda_h(j) \leq \lambda(j)$, which provides higher-order convergence rates of the LEB.

In addition to LEB, we provide a full error analysis of the source and eigenvalue problem under minimal $H^2$ regularity assumptions of the exact solution.
For the source problem associated with \eqref{def:pde}, we prove the Cea's type estimate
\begin{align}\label{ineq:a-priori}
    \|\nabla^2_\pw(u - R_h u_h)\|_\Omega + |u_h|_{s_h} \lesssim \min_{p \in P_{k+2}(\mathcal{T})} \|\nabla^2_\pw(u - p)\|_\Omega + \mathrm{osc}(f,\mathcal{T})
\end{align}
with the data oscillation $\mathrm{osc}(f,\mathcal{T})$ of a given right-hand side $f \in L^2(\Omega)$.
This is to be contrasted with the error analysis of \cite{DongErn2022}, where $H^{2+s}$ regularity of the solution with $s > 3/2$ is required. A reduction of the smoothness assumption can be achieved by employing $C^0$ ansatz spaces for the cell variable on regular simplicial meshes, cf.~\cite{DongErn2024}.

An important aspect of \eqref{ineq:a-priori} is the efficiency $|u_h|_{s_h} \lesssim $ of the stabilization $|u_h|_{s_h}$ of $u_h$, which contributes to a reliable and efficient error control
\begin{align}\label{ineq:a-posteriori}
    \|\nabla^2_\pw(u - R_h u_h)\|_\Omega \lesssim \eta \lesssim \|\nabla^2_\pw(u - R_h u_h)\|_\Omega + \mathrm{osc}(f,\T)
\end{align}
with the error estimator $\eta$ defined in \eqref{def:eta} below. Thus, this paper provides an extension of the analysis of \cite{ErnZanotti2020,BertrandCarstensenGraessleTran2023,CarstensenTran2024} from second-order to fourth-order problems.
The error analysis of the source problem can be extended to the eigenvalue problem using the arguments of \cite{Gallistl2015,CarstensenGraessleTran2024}.

For the sake of brevity, we only consider conforming simplicial meshes but mention that an extension to general polytopal meshes is straight-forward, cf.~\Cref{rem:polytopal-mesh}.
Finally, we note that parts of our error analysis are restricted to $n \in \{2,3\}$. The reason is the lack of conforming finite element spaces in higher space dimensions, whose degrees of freedom do not depend on second or higher derivatives -- a technical difficulty that may be solved in the future.

The remaining parts of this paper are organized as follows.
We start with the introduction of the numerical scheme for the source problem with details on the reconstruction operators and stabilizations in \Cref{sec:discretization}.
\Cref{sec:error-analysis} construct a continuous right-inverse $J_h$ of the interpolation $I_h$ and establishes the error control \eqref{ineq:a-priori}--\eqref{ineq:a-posteriori} as well as $L^2$ estimates.
An extension of the error analysis of \Cref{sec:error-analysis} to the eigenvalue problem with the LEB \eqref{ineq:LEB} is carried out in \Cref{sec:eigenvalue_problem}. Some numerical examples in \Cref{sec:num_examples} conclude this paper.

Standard notation for Lebesgue and Sobolev spaces applies throughout this paper.
For any set $M$, $(\bullet, \bullet)_M$ denotes the $L^2$ scalar product and $\|\bullet\|_M$ is the $L^2$ norm in $L^2(M)$. The notation $A \lesssim B$ means $A \leq CB$ with a constant $C$ independent of the mesh size and $A \approx B$ abbreviates $A \lesssim B$ and $B \lesssim A$. 

\section{Discretization}\label{sec:discretization}

In this section, we provide details on the numerical scheme for the source problem associated with \eqref{def:pde-var}, which seeks, for a given right-hand side $f \in L^2(\Omega)$, the solution $u \in H^2_0(\Omega)$ such that
\begin{align}\label{def:source-problem}
    (\nabla^2 u, \nabla^2 v)_{\Omega} = (f,v)_{\Omega} \quad\text{for any } v \in H^2_0(\Omega).
\end{align}
The HHO method for the eigenvalue problem is presented in \Cref{sec:eigenvalue_problem}, where the stabilization is slightly modified for explicit constants in the LEB \eqref{ineq:LEB}.

\subsection{Triangulation}
A $d$-simplex, $0 \leq d \leq n$, is the convex hull of $d+1$ points with positive $d$ dimensional Lebesgue measure.
A $d$ dimensional side of a simplex $T$ is the convex combination of $d+1$ vertices of $T$.
The set of $n-1$ (resp.~$n-2$) dimensional sides of $T$, called faces (resp.~edges) of $T$, is denoted by $\mathcal{F}(T)$ (resp.~$\mathcal{E}(T)$).
Let $\mathcal{T}$ be a regular triangulation of $\Omega$ into closed nonempty simplices so that $\cup \mathcal{T} = \overline{\Omega}$ and the intersection of two distinct simplices is either empty or exactly one lower-dimensional side.
The sets $\mathcal{F} \coloneqq \cup_{T \in \mathcal{T}} \mathcal{F}(T)$ and $\mathcal{E} \coloneqq \cup_{T \in \mathcal{T}} \mathcal{E}(T)$ consist of all $n-1$ and $n-2$ dimensional sides in $\mathcal{T}$.
Interior sides are collected in $\mathcal{F}(\Omega)$ and $\mathcal{E}(\Omega)$, while $\mathcal{F}(\partial \Omega) \coloneqq \mathcal{F}\setminus\mathcal{F}(\Omega)$ and $\mathcal{E}(\partial \Omega) \coloneqq \mathcal{E}\setminus\mathcal{E}(\Omega)$ are the set of boundary sides.

For any face $F \in \mathcal{F}$, the direction of its unit outer normal vector is fixed so that $\nu_F \coloneqq \nu_{\partial \Omega}|_F$ for each $F\in\mathcal{F}(\partial \Omega)$. The neighbors of an interior face $F \in \mathcal{F}(\Omega)$ are denoted by $T_\pm$ with the convention $\nu_{\partial T_+}|_F = \nu_F$. The jump of a piecewise function $v \in W^{1,1}(\mathrm{int}(T_\pm))$ reads $[v]_F \coloneqq v|_{T_+} - v|_{T_-} \in L^1(F)$. On boundary faces, the jump of any function is defined as its trace.

Given a $d$-simplex $M \subset \R^n$ of diameter $h_M \coloneqq \mathrm{diam}(M)$, the set $P_k(M)$ consists of all restrictions of polynomials of degree at most $k$ onto $M$ and the map $\Pi_M^k : L^1(M) \to P_k(M)$ denotes the $L^2$ orthogonal projection onto $P_k(M)$. In the special case $d = 0$ (i.e., $M$ is a vertex), $\Pi_M^k$ is the point evaluation.
The piecewise version of this for the triangulation $\mathcal{T}$ reads $\Pi_\mathcal{T}^k : L^1(\Omega) \to P_k(\mathcal{T})$, where $(\Pi_\mathcal{T}^k v)|_T \coloneqq \Pi_T^k v|_T$ for any $T \in \mathcal{T}$.
We define the projections $\Pi_\mathcal{F}^k$ and $\Pi_\mathcal{E}^k$ analogously.
The mesh size function $h_\mathcal{T}$ of the triangulation $\mathcal{T}$ is defined locally by $h_\mathcal{\mathcal{T}}|_T \coloneqq h_T$ for all $T \in \T$.
Given $f \in L^2(\Omega)$, the data oscillation of $f$ of degree $\ell \geq 0$ reads $\mathrm{osc}(f,\mathcal{T}) \coloneqq \|h_\mathcal{T}^2(1 - \Pi_\mathcal{T}^\ell) f\|_\Omega$.
The shape regularity $\varrho(T)$ of a simplex $T$ is the ratio of its diameter to the radius of the largest inscribed ball. This gives rise to the shape regularity $\varrho(\mathcal{T}) \coloneqq \max_{T \in \mathcal{T}} \varrho(T)$ of the triangulation $\mathcal{T}$.

The differential operators $\nabla^2_\pw$, $\nabla_\pw$, and $\Delta^2_\pw$ denote the piecewise
version of $\nabla^2$, $\nabla$ and $\Delta$ without explicit reference to $\T$.

\subsection{Local ansatz space and reconstructions}
Given $k \geq 0$ and a simplex $T\in \mathcal{T}$, the local HHO ansatz space reads
\begin{equation*}
	V_h(T) \coloneqq P_{\ell}(T) \times P_{m}(\mathcal{F}(T))\times P_k(\mathcal{F}(T))\times P_k(\mathcal{E}(T)),
\end{equation*}
with $\ell \geq \max\{k-2,1\}$, $m=\max\{k-1,0\}$, and the canonical interpolation $I_T : W^{2,1}(\mathrm{int}(T)) \to V_h(T)$, defined by
\begin{align*}
    v \mapsto (\Pi_T^\ell v, \Pi_{\mathcal{F}(T)}^m v, \Pi_{\mathcal{F}(T)}^k \partial_n v, \Pi_{\mathcal{E}(T)}^k v) \in V_h(T).
\end{align*}
We note that the choice $\ell \coloneqq k+2$ is relevant for computational LEB in \Cref{sec:eigenvalue_problem} below, while $\ell \geq 1$ is required for the approximation of the right-hand side.
The local reconstruction operator $R_T : V_h(T) \to P_{k+2}(T)$ maps $v_h = (v_T, v_{\mathcal{F}(T)}, \beta_{\mathcal{F}(T)}, v_{\mathcal{E}(T)}) \in V_h(T)$ onto $R_T v_h \in P_{k+2}(T)$ with
\begin{equation}\label{eq:local_restruct}
	\begin{aligned}
		(\nabla^2 R_T v_h, \nabla^2 p)_T & = (v_T,\Delta^2 p)_T - (v_{\mathcal{F}(T)}, \partial_n\Delta p)_{\partial T} + (\beta_{\mathcal{F}(T)}, \partial_{nn}p)_{\partial T} \\ &\quad - (v_{\mathcal{F}(T)},\Delta_t \partial_n p)_{\partial T} + \sum_{F \in \mathcal{F(T)}} (v_{\mathcal{E}(T)}, \partial_{tn} p\cdot \nu_{\partial F})_{\partial F}
	\end{aligned}
\end{equation}
for any $p \in P_{k+2}(T)$. This uniquely defines $R_T v_h$ up to the degrees of freedom associated with $P_1(T)$, the kernel of $\nabla^2$. The latter are fixed by
\begin{align}\label{def:rec_constr}
    \int_T \nabla R_T v_h \,\mathrm{d}x = \int_{\partial T} v_{\mathcal{F}(T)} \nu_{\partial T} \,\mathrm{d}s \quad\text{and}\quad
	\int_T R_T v_h \,\mathrm{d}x =	\int_T v_T \,\mathrm{d}x.
\end{align}
By construction of $R_T$, the following property holds.
\begin{lemma}[commuting]\label{lem:commuting}
    Any $v \in H^2(T)$ satisfies the $L^2$ orthogonality $\nabla^2 (v - R_T I_T v) \perp \nabla^2 P_{k+2}(T)$. In particular,
    \begin{align}\label{ineq:RI-best-appr}
        \|\nabla^2 (v - R_T I_T v)\|_T = \min_{p \in P_{k+2}(T)} \|\nabla^2 (v - p)\|_T.
    \end{align}
\end{lemma}
\begin{proof}
    The $L^2$ orthogonality follows immediately from \eqref{eq:ibp} and \eqref{eq:local_restruct}. This implies the best approximation property \eqref{ineq:RI-best-appr}.
\end{proof}
The local stabilization $s_T : V_h(T) \times V_h(T) \to \mathbb{R}$ of this paper reads, for any $u_h = (u_T, u_{\mathcal{F}(T)}, \alpha_{\mathcal{F}(T)}, u_{\E(T)}), v_h = (v_T, v_{\F(T)}, \beta_{\F(T)}, v_{\E(T)}) \in V_h(T)$,
\begin{align}\label{def:stab_loc}
    \begin{aligned}
        s_T(u_h,v_h) \coloneqq  &  h_T^{-4} (\Pi_T^{\ell}(u_{T} - R_T u_h), \Pi_{T}^{\ell}(v_T - R_T v_h))_{T}\\
    	&\quad + h_T^{-3} (\Pi_{\F(T)}^{m}(u_{\F(T)} - R_T u_h),\Pi_{\F(T)}^{m}(v_{\F(T)} - R_T v_h))_{\partial T}\\
        &\quad + h_T^{-1} (\Pi_{\F(T)}^k(\alpha_{\F(T)}-\partial_n R_T u_h),\Pi_{\F(T)}^k(\beta_{\F(T)} - \partial_n R_T v_h)_{\partial T}\\
    	&\quad + h_T^{-2} \sum_{E \in \E(T)} (\Pi_{E}^{k}(u_{\E(T)} -R_T u_h),\Pi_{E}^{k}(v_{\E(T)} - R_T v_h))_{E}.
    \end{aligned}
\end{align}
The stabilization is quasi-optimal in the following sense.
\begin{lemma}[optimality of $s_T$]\label{lem:best-appr-stab}
    Let a simplex $T$ be given. Any $v \in H^2(T)$ satisfies
    \begin{align*}
        \sqrt{s_T(I_T v, I_T v)} \lesssim \min_{p \in P_{k+2}(\T)} \|\nabla^2(v - p)\|_T.
    \end{align*}
\end{lemma}
\begin{proof}
    The constraints \eqref{def:rec_constr} give rise to
    \begin{align}\label{eq:vanishing_mean}
    	\begin{aligned}
    		\int_T R_T I_T v \dx &= \int_T \Pi_T^\ell v \dx = \int_T v \dx,\\
    		\int_T \nabla R_T I_T v \dx &= \int_{\partial T} \Pi_{\F(T)}^m v \nu_{\partial T} \ds = \int_{\partial T} v \nu_{\partial T} \ds = \int_T \nabla v \dx.
    	\end{aligned}
    \end{align}
    Hence, the Poincar\'e inequality imply
    \begin{align}
        h_T^{-2} \|v - R_T I_T v\|_T + h_T^{-1} \|\nabla(v - R_T I_T v)\|_T \lesssim \|\nabla^2(v - R_T I_T v)\|_T.
    \end{align}
    From this, the best approximation property of $L^2$ projections, the trace inequality, we infer
    \begin{align*}
        &s_T(I_T v, I_T v) \leq h_T^{-4} \|v - R_T I_T v\|_T^2 + h_T^{-3}\|v - R_T I_T v\|_{\partial T}^2\\
        &\quad+ h_T^{-1} \|\nabla(v - R_T I_T v)\|_{\partial T}^2 + \sum_{E \in \E(T)} h_T^{-2} \|v - R_T I_T v\|_E^2 \lesssim \|\nabla^2(v - R_T I_T v)\|_T^2.
    \end{align*}
    This and \eqref{ineq:RI-best-appr} conclude the proof.
\end{proof}
\subsection{Discrete problem}
The discrete ansatz space of $H^2_0(\Omega)$ reads
\begin{equation*}
	V_h \coloneqq P_\ell(\mathcal{T}) \times P_{m}(\mathcal{F}(\Omega)) \times P_k(\F(\Omega)) \times P_k(\E(\Omega))
\end{equation*}
with the local restriction 
\begin{align*}
    v_h|_T \coloneqq (v_\T|_T, v_\F|_{\partial T}, ((\nu_F \cdot \nu_{\partial T}|_F) \beta_\F|_{F})_{F \in \F(T)}, (v_\E|_E)_{E \in \E(T)}) \in V_h(T)
\end{align*}
in $T \in \T$
for any $v_h = (v_\T, v_\F, \beta_\F, v_\E) \in V_h$.
Here, $P_{m}(\mathcal{F}(\Omega))$ (resp.~$P_k(\E(\Omega))$) is the set of all functions $v_\F \in P_m(\F)$ (resp.~$v_\E \in P_k(\E)$) vanishing along boundary faces $v_\F|_F = 0$ for all $F \in \F(\partial \Omega)$ (resp.~boundary edges $v_\E|_E = 0$ for all $E \in \E(\partial \Omega)$) to model homogeneous boundary conditions. The global interpolation $I_h:H_0^2(\Omega) \rightarrow V_h$ is defined by
\begin{align*}
    v \mapsto (\Pi_{\mathcal{T}}^\ell v, \Pi_{\mathcal{F}}^m v, \Pi_{\mathcal{F}}^k (\nu_{\mathcal{F}}\cdot\nabla v), \Pi_{\mathcal{E}}^k v) \in V_h.
\end{align*}
We define the reconstructions $R_h$ and the stabilization $s_h$ locally by
\begin{align}
    (R_h v_h)|_T = R_T (v_h|_T) \quad\text{in } T \in \T \quad\text{and}\quad s_h(u_h,v_h) \coloneqq \sum_{T \in \T} s_T(u_h|_T, v_h|_T)
\end{align}
for any $u_h,v_h \in V_h$. The discrete problem seeks the unique solution $u_h \in V_h$ to
\begin{align}\label{def:discrete_problem}
    a_h(u_h,v_h) = (f,v_\T)_\Omega \quad\text{for any } v_h = (v_\T, v_\F, \beta_\F, v_\E) \in V_h
\end{align}
with the bilinear form
\begin{align*}
    a_h(u_h,v_h) \coloneqq (\nabla_\pw^2 R_h u_h, \nabla_\pw^2 R_h v_h)_\Omega + s_h(u_h,v_h),
\end{align*}
which induce the seminorms $\|\bullet\|_h \coloneqq \sqrt{a_h(\bullet,\bullet)}$ and $|\bullet|_{s_h} \coloneqq \sqrt{s_h(\bullet,\bullet)}$ in $V_h$.

\begin{remark}[relation to WG method of \cite{LiWangWangZhang2024}]
    If $\ell = k + 2$, the discrete Hessian $\partial^2_w$ in \cite{LiWangWangZhang2024} is reconstructed in the space $P_k(\mathcal{T})^{n \times n}$ of piecewise polynomials of degree at most $k$ using a similar formula to \eqref{eq:local_restruct}. It can be shown that the $L^2$ projection of the discrete Hessian $\partial_w^2 v_h$ from \cite{LiWangWangZhang2024} onto $\nabla^2_\pw P_{k+2}(\mathcal{T})$ is equal to $\nabla^2_\pw R_h v_h$ for any $v_h \in V_h$. The stabilization $s_h$ in \eqref{def:stab_loc} displays a similar structure to that of \cite{LiWangWangZhang2024}, but additionally involves the reconstruction operator $R_h v_h$.
\end{remark}
\begin{remark}[static condensation]
	The volume variables associated with the mesh $\T$ can be statically condensed. 
\end{remark}
\begin{table}[ht!]
    \centering
    \begin{tabular}{|c|cccc|c|}
        \hline
        unknowns & cell & face & grad & edge & $k$\\
        \hline
        \cite{DongErn2022} (2D) & $k+2$ & $k+1$ & $k$ & -- & $\geq 0$\\
        \cite{DongErn2022} (3D) & $k+2$ & $k+2$ & $k$ & -- & $\geq 0$\\
        \hline
        present & $\geq \max\{k-2,1\}$ & $\max\{k-1,0\}$ & $k$ & $k$ & $\geq 0$\\
        \hline
    \end{tabular}
    \captionsetup{width=1\linewidth}
    \caption{Degrees of freedom of HHO methods in \cite{DongErn2022} and the present method.}
    \label{tab:dof}
\end{table}
\begin{remark}[computational cost]\label{rem:computational-cost}
    In 2d with the choice $k \geq 1$, 
    we have $(2k+1)|\mathcal{F}(\Omega)| + |\mathcal{E}(\Omega)|$ global degrees of freedom.
    For general meshes into $j$-sided polygons ($j \geq 3$) and small mesh sizes (where the number of boundary faces are neglectable), $|\mathcal{T}| \sim 2|\mathcal{F}|/j$.
    This and the Euler formula $|\mathcal{E}| + |\mathcal{T}| - 2 = |\mathcal{F}|$ lead to approximately $(2k+1)|\mathcal{F}| + (j-2)|\mathcal{F}|/j$ global degrees of freedom, whereas \cite{DongErn2022} utilizes approximately $(2k+3)|\mathcal{F}|$ degrees of freedom in 2d, cf.~\Cref{tab:dof}. Thus, our method employs $(j+2)|\mathcal{F}|/j$ less global degrees of freedom, which becomes more significant for small $k$.
    In 3d, the situation is less clear due to the unknown number of simplices sharing an edge.

    On the other hand,
    the introduction of degrees of freedom on $n-2$ dimensional sides of an element leads to a larger stencil comparing to \cite{DongErn2022}, cf.~\Cref{sec:num_ex_bihar} for a visualization of the sparsity pattern of the stiffness matrix on a uniform mesh.
    Furthermore, numerical results on the computational times for different HHO methods in \cite{DongErn2022} suggest a slightly higher computational cost when the reconstruction operator $R_h$ is involved in the stabilization.
    Overall, we expect that our method is computationally more expensive than that of \cite{DongErn2022}.

    However,
    we mention that $R_h$ is a crucial ingredient for explicit constants in the guaranteed lower eigenvalue bounds derived below. By setting the degrees of freedom associated with the kernel of the Hessian in \eqref{def:rec_constr}, we have access to the Poincar\'e inequality, leading to constants independent of the polynomial degree $k$ in \Cref{thm:LEB}.
    Another theoretical advantage is the access to quasi-optimal error estimates in \Cref{sec:error-analysis} for general polytopal meshes, cf.~\Cref{rem:polytopal-mesh}.
\end{remark}

\begin{theorem}[existence and uniqueness of solutions]\label{thm:wellposedness}
    The bilinear form $a_h$ is a scalar product in $V_h$. In particular, there exists a unique solution $u_h$ to \eqref{def:discrete_problem}.
\end{theorem}
\begin{proof}
    We provide only a proof for the positive definiteness of the induced norm $\|\bullet\|_h$. Suppose that $\|v_h\|_h = 0$ for some $v_h = (v_\T, v_\F, \beta_\F, v_\E) \in V_h$, then $\nabla^2_\pw R_h v_h = 0$ and so, $R_h v_h$ is a piecewise affine function. Since the stabilization vanishes,
    $\|\Pi_F^k(\beta_\F - \nabla R_h v_h \cdot \nu_F)\|_F = \|\beta_\F - \nabla R_h v_h \cdot \nu_F\|_F = 0$ for any $F \in \F$ enforces the continuity of the gradient of $R_h v_h$ in the normal direction. 
    The jump $[R_h v_h]_F$ along any interior face $F \in \F$ is an affine function on $F$ and vanishes at all $n$ midpoints of the edges of $F$ from 
    $\|\Pi_E^k(v_\E - R_h v_h)\|_E = 0$ for any $E \in \E$. Hence, $R_h v_h$ is continuous. This and the continuity of the gradient in normal directions prove that $R_h v_h$ is a (global) affine function. The boundary data imposed by $V_h$ and the vanishing stabilization conclude $R_h v_h = 0$ as well as $v_h = 0$.
\end{proof}

We introduce the Galerkin projection $G_h \coloneqq R_h \circ I_h : V \to P_{k+2}(\T)$. From \Cref{lem:commuting}, any $v \in V$ satisfies the $L^2$ orthogonality
\begin{align}\label{eq:commuting-G}
	\nabla_\pw^2(v - G_h v) \perp \nabla_\pw^2 P_{k+2}(\T)
\end{align}
and the best-approximation property
\begin{align}\label{eq:best-appr}
	\|\nabla_\pw^2(v - G_h v)\|_\Omega = \min_{p \in P_{k+2}(\T)} \|\nabla_\pw^2(v - p)\|_\Omega.
\end{align}
Recall the vanishing mean properties \eqref{eq:vanishing_mean} arising from the constraints \eqref{def:rec_constr}.
The Poincar\'e inequality with the constant $1/\pi$ in convex domains \cite{Bebendorf2003} implies
\begin{align}\label{ineq:v-Gv-Poincare}
	\|h_\T^{-2}(v - G_h v)\|_\Omega/\pi^2 \leq \|h_\T^{-1}\nabla_\pw(v - G_h v)\|_\Omega/\pi \leq \|\nabla_\pw^2(v - G_h v)\|_\Omega.
\end{align}
The frequently employed bound
\begin{align}\label{ineq:v-Piv}
	\|h_\T^{-2}(1 - \Pi_\T^{k+2}) v\|_\Omega \leq \|h_\T^{-2}(v - G_h v)\|_\Omega \leq \pi^2\|\nabla_\pw^2(v - G_h v)\|_\Omega
\end{align}
follows from the best-approximation property of the $L^2$ projection and is stated here for later reference.
\begin{remark}[alternative side conditions for $R_T$]
    In \cite{DongErn2022}, the orthogonality
    $$(R_T v_h, p)_T = (v_T,p)_T$$ for any $p \in P_1(T)$ is utilized instead of \eqref{def:rec_constr} to fix the degrees of freedom associated with $P_1(T)$ in the definition of $R_T$. A similar relation holds for the Galerkin projection (called $H^2$-elliptic projection therein). The major part of this paper is not restricted to the choice \eqref{def:rec_constr} but it is required for explicit constants in \eqref{ineq:v-Gv-Poincare}, contributing to LEB in \Cref{sec:LEB} below.
\end{remark}

\section{Error analysis of the source problem}\label{sec:error-analysis}
This section establishes the a~priori and a~posteriori error estimates \eqref{ineq:a-priori}--\eqref{ineq:a-posteriori} for the source problem.

\subsection{Right-inverse}\label{sec:right-inverse}
The error analysis of this paper relies on the construction of a right-inverse $J_h : V_h \to V$ of the interpolation $I_h$ in the spirit of \cite{VeeserZanotti2019,ErnZanotti2020} using averaging and correction techniques, cf.~\cite{Gallistl2015} for the Morley FEM and \cite{CarstensenKhotPani2023,KhotMoraRuizBaier2025} for virtual elements.

\begin{lemma}[right-inverse]\label{lem:right-inverse}
    There exists a continuous linear operator $J_h : V_h \to V$ with $I_h J_h = \mathrm{Id}$ in $V_h$.
\end{lemma}

\begin{remark}[extension to other hybridizable methods]
    The construction of this right-inverse operator can be adapted to other frameworks. For the method in \cite{LiWangWangZhang2024}, the discrete space and interpolation operator coincide with $V_h(T)$ and $I_h$ for $l=k+2$, respectively, rendering its construction straightforward. For the three-dimensional case in \cite{DongErn2022}, the construction follows by excluding the degrees of freedom on edges and setting $r =  \max\{\ell,m,k\} + 11$ in the subsequent argument. However, it is worth noting that for the two-dimensional case in \cite{DongErn2022}, there does not exist a right-inverse operator $J_h$ such that $I_h J_h= \mathrm{Id}$, where $I_h$ denotes the global reduction operator. This is due to the fact that $I_h$ in the 2D setting of \cite{DongErn2022} is not surjective onto the discrete space,  as the edge-based components of $I_h u$ for any $u\in H^2(\Omega)$ are continuous at the vertices while the discrete space permits nodal discontinuities.
\end{remark}
\begin{proof}
For the sake of brevity, we only provide details in three space dimensions because the 2D case is similar but simpler.\\[-0.5em] 

\noindent\emph{Construction.} Averaging techniques in $C^1$ finite element spaces on the Worsey-Farin split of $\T$ \cite{WorseyFarin1987,GuzmanLischkeNeilan2022} lead to a linear operator $\mathcal{A}_h : V_h \to V$ with
the typical local approximation property
\begin{align}\label{ineq:appr-averaging}
    \|h_T^{-2} \delta\|^2_T
    \lesssim \sum_{F \in \F, F \cap T \neq \emptyset} (h_F^{-3} \|[R_h v_h]_F\|^2_F + h_F^{-1} \|[\partial_nR_h v_h]_F\|^2_F)
\end{align}
for any $T \in \T$ and $\delta \coloneqq R_h v_h - \mathcal{A}_h v_h$, cf.~\cite{GallistlTian2024} for further details.
To ensure the right-inverse property, we consider a correction operator $S_h$
by setting appropriate degrees of freedom of $C^1$ conforming elements as an alternative approach to bubble function techniques.
Let $r = \max\{\ell,m,k\} + 10$.
For any $w=(w_{\mathcal{T}}, w_{\mathcal{F}}, \gamma_{\F}, w_\E)$ in the product space $L^2(\T) \times L^2(\cup\F) \times L^2(\cup\F) \times L^2(\cup \E)$, there exists a function $S_h w \in P_r(\T) \cap H^2_0(\Omega)$ with
\begin{align}\label{def:correction}
    \begin{aligned}
        &\int_E S_h w p \dt = \int_E w_\E p \dt \quad\text{for any } p \in P_{r-10}(E) \text{ along every edge } E \in \E(\Omega),\\
        &\int_F S_h w p \ds = \int_F w_\F p \ds \quad\text{for any } p\in P_{r-9}(F) \text{ on every face } F \in \F(\Omega),\\
        &\int_F \partial_n S_h w p \ds = \int_F \gamma_\F p \ds \quad\text{for any } p\in P_{r-7}(F) \text{ on every face } F \in \F(\Omega),\\
        &\int_T S_h w p \dx = \int_T w_\T p \dx \quad\text{for any } p \in P_{r-8}(T) \text{ in every } T \in \T.
    \end{aligned}
\end{align}
In fact, these weights are a subset of the degrees of freedom of the $C^1$ conforming finite element $P_r(\T) \cap H^2(\Omega)$ \cite{Zhang2009}, cf.~\Cref{appendix} for further details. By setting the remaining degrees of freedom to zero, we obtain $S_h w \in V$ and, from equivalence of norms in finite dimensional spaces and scaling arguments,
\begin{align}\label{ineq:bound-correction}
	\begin{aligned}
		&h_T^{-4}\|S_h w\|_T^2 \lesssim h_T^{-4}\|w_\T\|_T^2\\
		&\qquad + \sum_{F \in \F(T)} (h_F^{-3} \|w_\F\|_F^2 + h_F^{-1} \|\gamma_\F\|^2_F) + \sum_{E \in \E(T)} h_E^{-2} \|w_\E\|_E^2.
	\end{aligned}
\end{align}
Define $J_h : V_h \to V$ for any $v_h = (v_\T, v_\F, \beta_\F, v_\E) \in V_h$ by
\begin{align*}
    J_h v_h \coloneqq \mathcal{A}_h v_h + S_h w \in V
\end{align*}
with the choice $w = v_h - I_h \mathcal{A}_h v_h$.
By construction of $S_h$ in \eqref{def:correction}, $I_h J_h = \mathrm{Id}$ in $V_h$.\\[-0.5em]
 
\noindent\emph{Approximation property of $J_h$.}
Given a simplex $T \in \T$, \eqref{ineq:bound-correction}, the best approximation of $L^2$ projections, trace, and inverse inequalities imply
\begin{align*}
	&h_T^{-4}\|S_hw\|_T^2 \lesssim s_T(v_h,v_h) + h_T^{-4}\|\Pi_T^\ell \delta\|_T^2 + \sum_{E \in \E(T)} h_E^{-2}\|\Pi_F^k \delta|_T\|_E^2\\
    &\quad+ \sum_{F \in \F(T)} (h_F^{-3}\|\Pi_F^m \delta|_T\|_F^2 + h_F^{-1}\|\Pi_F^k\partial_n \delta|_T\|_F^2) \lesssim s_T(v_h,v_h) + h_T^{-4}\|\delta\|_T^2.
\end{align*}
This, $\|(J_h - R_h) v_h\|_T \leq \|\delta\|_T + \|S_hw\|_T$ from a triangle inequality, \eqref{ineq:appr-averaging}, and the sum over all simplices establish
\begin{align}\label{ineq:bound-J-R}
    &\|h_\T^{-2}(J_h - R_h) v_h\|_\Omega^2 \lesssim s_h(v_h,v_h)\nonumber\\
    &\qquad+ \sum_{F \in \F} (h_F^{-3} \|[R_h v_h]_F\|^2_F + h_F^{-1} \|[\partial_nR_h v_h]_F\|^2_F).    
\end{align}

\noindent\emph{Continuity of $J_h$.}
For any $F \in \F$, the triangle and Poincar\'e inequality show
\begin{align}\label{ineq:bound_jump_R}
    h_F^{-3/2}&\|[R_h v_h]_F\|_F \lesssim h_F^{-3/2}\|\Pi_F^0[R_h v_h]_F\|_F + h_F^{-1/2}\|[\partial_t R_h v_h]_F\|_F\nonumber\\
    &\lesssim h_F^{-3/2}\|\Pi_F^0[R_h v_h]_F\|_F + h_F^{-1/2}\|\Pi_F^0 [\partial_t R_h v_h]_F\|_F + h_F^{1/2}\|[\partial_{tt} R_h v_h]_F\|_F.
\end{align}
An integration by parts on the (two-dimensional) face $F$, the Cauchy, and discrete trace inequality establish, for any $p \in P_0(F)^2$, that
\begin{align*}
    \int_F \Pi_F^0 [\partial_t R_h v_h]_F p \ds &= \int_F [\partial_t R_h v_h]_F p \ds = \sum_{E \in \E(F)} \int_E [R_h v_h]_F p \cdot \nu_{\partial F}|_E \dt\\
    &\lesssim \Big(\sum_{E \in \E(F)} h_E^{-1}\|\Pi_E^k[R_h v_h]_F\|_E^2\Big)^{1/2}\|p\|_F,
\end{align*}
whence $h_F^{-1}\|\Pi_F^0 [\partial_t R_h v_h]_F\|_F^2 \lesssim \sum_{E \in \E(F)} h_E^{-2}\|\Pi_E^k [R_h v_h]_F\|_E^2$. If $F \in \F(\Omega)$ is an interior face with the neighbors $T_\pm$, then $\|\Pi_E^k [R_h v_h]_F\|_E \leq \|\Pi_E^k(R_h v_h|_{T_+} - v_\E)\|_E + \|\Pi_E^k(R_h v_h|_{T_-} - v_\E)\|_E$. If $F \in \F(\partial \Omega)$, then $E \in \E(\partial \Omega)$ and so, $\|\Pi_E^k[R_h v_h]_F\|_E = \|\Pi_E^k(R_h v_h - v_\E)\|_E$ from homogeneous boundary conditions. Thus,
\begin{align}\label{ineq:proof-cont-J-dt}
    h_F^{-1}\|\Pi_F^0 \partial_t [R_h v_h]_F\|_F^2 \lesssim \sum_{K \in \T, F \subset K} s_K(v_h,v_h).
\end{align}
It is straight-forward to verify $\Pi_F^0 \circ \Pi_F^m = \Pi_F^0$ and so, the best approximation property of $L^2$ projections
and the triangle inequality provide
\begin{align}\label{ineq:proof-cont-J-jump}
    \|\Pi_F^0[R_h v_h]_F\|_F^2 \leq \|\Pi_F^m[R_h v_h]_F\|_F^2 \lesssim \sum_{K \in \T, F \subset K} \|\Pi_F^m(R_h v_h|_K - v_\F)\|_F^2.
\end{align}
The combination of \eqref{ineq:bound_jump_R}--\eqref{ineq:proof-cont-J-jump} with a discrete trace inequality results in
\begin{align}\label{ineq:jump_R}
    h_F^{-3}\|[R_h v_h]_F\|_F^2 \lesssim \sum_{K \in \T, F \subset K} (s_K(v_h,v_h) + \|\nabla^2 R_h v_h\|_K^2)
\end{align}
We proceed as above to bound the normal gradient jump of $R_h v_h$ by
\begin{align}\label{ineq:normal_jump_R}
    h_F^{-1} \|[\partial_n R_h v_h]_F\|_F^2 &\lesssim h_F^{-1} \|\Pi_F^0[\partial_n R_h v_h]_F\|_F^2 + h_{F}\|[\partial_{tn} R_h v_h]_F\|_F^2\nonumber\\
    &\lesssim
    \sum_{K \in \T, F \subset K} (s_K(v_h,v_h) + \|\nabla^2 R_h v_h\|_K^2).
\end{align}
This, \eqref{ineq:bound-J-R}, \eqref{ineq:jump_R}, and inverse estimates conclude the continuity of $J_h$.
\end{proof}
\begin{remark}[quasi-best approximation error]
	In \Cref{lem:right-inverse}, let $v_h \coloneqq I_h v \in V_h$ for some $v \in V$. Since $[v]_F = 0$ and $\partial_n [v]_F = 0$ along any face $F \in \F$, \eqref{ineq:bound-J-R} and the trace inequality imply (recall $G_h = R_h \circ I_h$)
	\begin{align*}
		\mathrm{LHS}^2
		&\lesssim \sum_{F \in \F} (h_F^{-3} \|[G_h v]_F\|^2_F + h_F^{-1} \|\partial_n [G_h v]_F\|^2_F) + |I_h v|_{s_h}^2\\
		&\lesssim \|h_\T^{-2}(v - G_h v)\|_\Omega^2 + \|h_\T^{-1}\nabla_\pw(v - G_h v)\|_\Omega^2 + \|\nabla^2_\pw(v - G_h v)\|_\Omega^2 + |I_h v|_{s_h}^2
	\end{align*}
	with the left-hand side
	\begin{align*}
		\mathrm{LHS} \coloneqq \|h_\T^{-2}(J_h - R_h) I_h v\|_\Omega .
	\end{align*}
	This, inverse inequality, \eqref{ineq:v-Gv-Poincare}, and \Cref{lem:best-appr-stab} conclude the quasi-best approximation estimate
	\begin{equation}\label{ineq:quasi-best-J-R}
    \begin{aligned}
        \|h_\T^{-2}(J_h - R_h) I_h v\|_\Omega &+ \|h_\T^{-1}\nabla_\pw(J_h - R_h) I_h v\|_\Omega + \|\nabla_\pw^2(J_h - R_h) I_h v\|_\Omega \\ & \lesssim \|\nabla_\pw^2(v - G_h v)\|_\Omega = \min_{p \in P_{k+2}(\T)} \|\nabla_\pw^2(v - p)\|_\Omega.
    \end{aligned}
	\end{equation}
\end{remark}
\begin{remark}[quasi interpolation]
	We note that $J_h \circ I_h : V \to V$ defines an interpolation into $C^1$ conforming finite element spaces with the approximation property
	\begin{align}\label{ineq:quasi-interpolation}
		\|h_\T^{-2}(v - J_h I_h v)\|_\Omega + \|h_\T^{-1}\nabla (v - J_h I_h v)\|_\Omega + \|\nabla^2(v - J_h I_h v)\|_\Omega&\nonumber\\
		\lesssim \|\nabla_\pw^2(v - G_h v)\|_\Omega \leq \|\nabla^2 v\|_\Omega&
	\end{align}
	from \eqref{ineq:quasi-best-J-R} and a triangle inequality. 
\end{remark}

\subsection{A~priori}
Let $u \in V$ be the solution to the source problem \eqref{def:source-problem} with a given right-hand side $f \in L^2(\Omega)$.
\begin{theorem}[a~priori]\label{thm:a-priori}
    The discrete solution $u_h \in V_h$ to \eqref{def:discrete_problem} satisfies
    \begin{equation}\label{ineq:a-priori-thm}
        \|I_h u - u_h\|_{h} + |u_h|_{s_h} \lesssim \min_{p \in P_{k+2}(\mathcal{T})}\|\nabla^2_\pw(u - p)\|_{\Omega} + \mathrm{osc}(f,\mathcal{T})
    \end{equation}
\end{theorem}


\begin{proof}
    Abbreviate $e_h = I_h u - u_h \in V_h$. Then
    \begin{align*}
         \|\nabla_\pw^2 R_h e_h\|_\Omega^2 &= (\nabla^2 u,\nabla^2_\pw (R_h e_h - J_h e_h))_\Omega\\
         &\quad+ (\nabla^2 u,\nabla^2 J_h e_h)_\Omega - (\nabla^2_\pw R_h u_h,\nabla^2_\pw R_h e_h)_\Omega.
    \end{align*}
    The variational formulations on the continuous and discrete level imply
    \begin{align*}
        (\nabla^2 u,\nabla^2 J_h e_h)_\Omega - (\nabla^2_\pw R_h u_h,\nabla^2_\pw R_h e_h)_\Omega = (f, J_h e_h - e_\T)_\Omega + s_h(u_h,e_h),
    \end{align*}
    where $e_\T$ is the first component of $e_h$ associated with cell degrees of freedom.
    A straightforward calculation provides the identity
    \begin{align*}
        s_h(u_h,e_h) = \frac{1}{2}(|I_h u|_{s_h}^2 - |e_h|^2_{s_h} - |u_h|^2_{s_h}).
    \end{align*}
    The combination of the three previously displayed formula leads to
    \begin{align}\label{ineq:proof_a_priori}
        &\|\nabla_\pw^2 R_h e_h\|_\Omega^2 + \frac{1}{2}(|e_h|^2_{s_h} + |u_h|^2_{s_h}) \nonumber\\
        &\qquad = (\nabla^2 u,\nabla^2_\mathrm{pw} (R_h e_h - J_h e_h))_\Omega + \frac{1}{2}|I_h u|_{s_h}^2 + (f, J_h e_h - e_\T)_\Omega.
    \end{align}
    The data oscillation arises from the $L^2$ orthogonality $J_h e_h - e_\T \perp P_{\ell}(\T)$ in \Cref{lem:right-inverse}, the Cauchy, and Poincar\'e inequality in
    \begin{align}\label{ineq:osc}
        \begin{aligned}
        	(f, J_h e_h - e_\T)_\Omega
        	&= ((1-\Pi^{\ell}_\T)f,(1 - \Pi_\T^1)J_h e_h)_\Omega\\
        	&\lesssim \operatorname{osc}(f,\mathcal{T})\|\nabla^2 J_h e_h\|_\Omega.
        \end{aligned}
    \end{align}
    The Galerkin projection $G_h = R_h \circ I_h$ satisfies
    $G_h J_h e_h = R_h e_h$ from the right-inverse property of $J_h$. Hence, the $L^2$ orthogonality
    \begin{align}\label{eq:orth-R-J}
    	\nabla^2_\mathrm{pw} (R_h e_h - J_h e_h) \perp \nabla^2_\mathrm{pw} P_{k+2}(\mathcal{T})
    \end{align}
    follows from \Cref{lem:commuting}. This and the continuity of $J_h$ in \Cref{lem:right-inverse} show
    \begin{align}\label{ineq:proof_a_priori_stab}
        (\nabla^2 u,\nabla^2_\mathrm{pw} (R_h e_h - J_h e_h))_\Omega &= 
        - (\nabla^2_\pw (u - G_h u),\nabla^2 J_h e_h)_\Omega\nonumber\\ &\lesssim \|\nabla_\pw^2 (u - G_h u)\|_\Omega \|e_h\|_{h}.
    \end{align}
    The combination of \eqref{ineq:proof_a_priori}--\eqref{ineq:osc} with \eqref{ineq:proof_a_priori_stab} and the quasi-best approximation of the stabilization in \Cref{lem:best-appr-stab} concludes the proof.
\end{proof}

To derive $L^2$ error estimates by duality techniques, we assume the existence of an index $2 \geq s > 0$ of elliptic regularity in polyhedral domains, that is any solution $z \in V$ to
$\Delta^2 z = g$ in $\Omega$
for some $g \in L^2(\Omega)$ satisfies $z \in H^{2+s}(\Omega)$ with
\begin{align}\label{ineq:elliptic-reg}
	\|z\|_{H^{2+s}(\Omega)} \lesssim \|g\|_\Omega.
\end{align}
In 2d, it is known from \cite{BlumRannacher1980,MelzerRannacher1980} that $s>1/2$. An explicit computation of $s$ in 3d is rather involved, we refer to \cite{Dauge1988} for an abstract theory.
\begin{theorem}[$L^2$ error]\label{thm:L2}
	If $\ell \geq 2$, the discrete solution $u_h \in V_h$ to \eqref{def:discrete_problem} and $t \coloneqq \min\{s,\ell-1\}$ satisfy
	\begin{align*}
		\|\Pi_\T^\ell(u - u_\T)\|_\Omega \lesssim h^{t} (\min_{p \in P_{k+2}(\mathcal{T})}\|\nabla^2_\pw(u - p)\|_{\Omega} + \mathrm{osc}(f,\T)).
	\end{align*}
\end{theorem}

\begin{proof}
    We adopt the abbreviation $e_h = (e_\T, e_\F, \gamma_\F, e_\E) = I_h u - u_h \in V_h$ from above. Let $z\in V$ solve $\Delta^2 z = e_{\mathcal{T}}$ in $\Omega$, i.e.,
    \begin{equation*}
    	(\nabla^2 z,\nabla^2v)_{\Omega} = (e_\T, v)_{\Omega} \quad \text{for any } v \in V.
    \end{equation*}
    Since $\Pi_\T^\ell J_h e_h = e_\T$ from $I_h \circ J_h = \mathrm{Id}$ in $V_h$, this implies the split
    \begin{align}\label{eq:proof-L2-split}
    	\|e_\T\|_\Omega^2 &= (e_\T, J_h e_h)_{\Omega} = (\nabla^2 z,\nabla^2 J_h e_h)_{\Omega}\nonumber\\
    	& = (\nabla^2 z,\nabla^2_\pw R_h e_h)_{\Omega} + (\nabla^2 z, \nabla^2_\pw (J_h e_h - R_h e_h))_{\Omega}.
    \end{align}
    The $L^2$ orthogonality \eqref{eq:orth-R-J}, a Cauchy inequality, and the continuity of $J_h$ prove
    \begin{align}\label{ineq:proof-L2-1}
    	(\nabla^2 z, \nabla^2_\pw (J_h e_h - R_h e_h))_{\Omega} & = (\nabla^2_\pw (z- G_h z),\nabla^2_\pw J_h e_h)_{\Omega}\nonumber\\
    	&\lesssim \|\nabla^2_\pw (z - G_h z)\|_\Omega \|e_h\|_h.
    \end{align}
    Linear algebra shows the split
    \begin{align*}
    	(\nabla^2 z,\nabla^2_\pw &R_h e_h)_{\Omega} = (\nabla^2 z, \nabla^2_\pw G_h u)_\Omega - (\nabla^2_\pw G_h z, \nabla^2_\pw R_h u_h)_\Omega\\
    	&= (\nabla^2 z, \nabla^2 u)_\Omega - (\nabla^2_\pw G_h z, \nabla^2_\pw R_h u_h)_\Omega - (\nabla^2 z, \nabla^2_\pw(u - G_h u))_\Omega.
    \end{align*}
    The first term on the right-hand side is equal to
    \begin{align*}
    	 (\nabla^2_\pw(z - G_h z), \nabla^2 u)_\Omega + (\nabla_\pw^2(G_h z - J_h I_h z), \nabla^2 u)_\Omega + (\nabla^2 J_h I_h z, \nabla^2 u)_\Omega.
    \end{align*}
    This, the previously displayed formula, the variational formulations \eqref{def:source-problem} and \eqref{def:discrete_problem}, and $(\nabla^2 z, \nabla^2_\pw(u - G_h u))_\Omega = (\nabla^2_\pw(z - G_h z), \nabla^2 u)_\Omega$ from \eqref{eq:commuting-G} provide
    \begin{align}\label{eq:proof-L2-identity}
    	(\nabla^2 z,\nabla^2_\pw R_h e_h)_{\Omega} &= (f, J_h I_h z - \Pi_\T^\ell z)_\Omega\nonumber\\
    	&\qquad + s(I_h z, u_h) + (\nabla_\pw^2(G_h z - J_h I_h z), \nabla^2 u)_\Omega.
    \end{align}
    The $L^2$ orthogonality $J_h I_h z - \Pi_\T^\ell z \perp P_{\ell}(\T)$ from $I_h \circ J_h = \mathrm{Id}$ in $V_h$ and \eqref{ineq:quasi-interpolation} imply
    \begin{align*}
    	(f, J_h I_h z - \Pi_\T^\ell z)_\Omega &= ((1 - \Pi_\T^\ell)f, J_h I_h z-z)_\Omega + ((1 - \Pi_\T^\ell)f, (1 - \Pi_\T^\ell) z)_\Omega\\
    	&\lesssim \mathrm{osc}(f,\T)(\|\nabla_\pw^2(z - G_h z)\|_\Omega + \|h_\T^{-2}(1-\Pi_\T^\ell)z\|_\Omega).
    \end{align*}
    Hence, \eqref{eq:proof-L2-identity}, the $L^2$ orthogonality $\nabla_\pw^2(G_h z - J_h I_h z) \perp \nabla_\pw^2 P_{k+2}(\T)$ from \Cref{lem:commuting}, the Cauchy inequality, \Cref{lem:best-appr-stab}, and \eqref{ineq:quasi-best-J-R} show
    \begin{align*}
    	(\nabla^2 z,\nabla^2_\pw R_h e_h)_{\Omega} \lesssim (\mathrm{osc}(f,\T) + |u_h|_{s_h} + \|\nabla_\pw^2(u - G_h u)\|_\Omega)&\\
        \times (\|\nabla_\pw^2(z - G_h z)\|_\Omega + \|h_\T^{-2}(1-\Pi_\T^\ell)z\|_\Omega)&.
    \end{align*}
    The combination of this with \eqref{eq:proof-L2-split}--\eqref{ineq:proof-L2-1} results in
    \begin{align*}
    	\|e_\T\|^2_\Omega \lesssim (\mathrm{osc}(f,\T) + \|e_h\|_h + |u_h|_{s_h} + \|\nabla_\pw^2(u - G_h u)\|_\Omega)&\\
        \times (\|\nabla_\pw^2(z - G_h z)\|_\Omega + \|h_\T^{-2}(1-\Pi_\T^\ell)z\|_\Omega)&.
    \end{align*}
    The assertion then follows from \Cref{thm:a-priori} and the elliptic regularity \eqref{ineq:elliptic-reg}.
\end{proof}

\subsection{A~posteriori}
In the following, we derive \eqref{ineq:a-posteriori} with the a~posteriori error estimator
\begin{align}\label{def:eta}
	\eta^2 \coloneqq \mu^2 + |u_h|^2_{s_h} + \mathrm{osc}(f,\mathcal{T})^2
\end{align}
with the jump contributions
\begin{align}\label{def:mu}
	\mu^2 \coloneqq \sum_{F \in \F} (h_F^{-3} \|[R_h u_h]_F\|^2_F + h_F^{-1} \|[\partial_n R_h u_h]_F\|_F^2).
\end{align}
\begin{theorem}[a~posteriori]\label{thm:a-posteriori}
    The discrete solution $u_h \in V_h$ to \eqref{def:discrete_problem} satisfies \eqref{ineq:a-posteriori}.
\end{theorem}

\begin{proof}
Throughout this proof, we abbreviate $e \coloneqq u - R_h u_h$.\\[-0.5em]

\noindent\emph{Reliability.~}
The proof departs from the error decomposition \cite[Theorem 1]{CarstensenGallistlHu2013}
\begin{align}\label{eq:proof-a-post-split}
	\|\nabla_\pw^2 e\|_\Omega^2 = \sup_{v \in V\setminus\{0\}} \frac{(\nabla_\pw^2 e, \nabla^2 v)_\Omega^2}{\|\nabla^2 v\|_\Omega^2}
	+ \min_{\varphi \in V}\|\nabla^2_\pw (\varphi - R_h u_h)\|_\Omega^2.
\end{align}
A bound for the conformity error $\min_{\varphi \in V}\|\nabla^2_\pw (\varphi - R_h u_h)\|_\Omega$ is given by the choice $\varphi \coloneqq \mathcal{A}_h u_h \in V$ in \eqref{ineq:appr-averaging}, namely,
\begin{align}\label{ineq:proof-a-post-conformity-error}
	\min_{\varphi \in V}\|\nabla^2_\pw (\varphi - R_h u_h)\|_\Omega^2 \lesssim \sum_{F \in \F} (h_F^{-3}\|[R_h u_h]_F\|_F^2 + h_F^{-1}\|[\partial_n R_h u_h]_F\|_F^2).
\end{align}
It remains to bound the consistency error
\begin{align}\label{ineq:proof-a-post-consistency-error}
	\sup_{v \in V\setminus\{0\}} \frac{(\nabla_\pw^2 e, \nabla^2 v)_\Omega}{\|\nabla^2 v\|_\Omega} \lesssim \eta.
\end{align}
We infer from \Cref{lem:commuting}, the variational formulations \eqref{def:source-problem}, and \eqref{def:discrete_problem} that
\begin{align}\label{eq:proof-a-post-case-2}
	(\nabla_\pw^2 e, \nabla^2 v)_\Omega &= (\nabla^2 u, \nabla^2 v)_\Omega - (\nabla_\pw^2 R_h u_h, \nabla_\pw^2 G_h v)_\Omega\nonumber\\
	&= (f, (1 - \Pi_\T^\ell) v)_\Omega + s_h(u_h, I_h v).
\end{align}
Since $\ell \geq 1$, $(f, (1 - \Pi_\T^\ell) v)_\Omega = ((1 - \Pi_\T^\ell) f,  (1 - \Pi_\T^1) v)_\Omega \lesssim \mathrm{osc}(f,\T)\|\nabla^2 v\|_\Omega$. This, \eqref{eq:proof-a-post-case-2}, the Cauchy inequality, and \Cref{lem:best-appr-stab} lead to
\begin{align*}
	(\nabla_\pw^2 e, \nabla^2 v)_\Omega \lesssim (\mathrm{osc}(f,\T) + |u_h|_{s_h})\|\nabla^2 v\|_\Omega,
\end{align*}
which implies \eqref{ineq:proof-a-post-consistency-error}.
This and \eqref{eq:proof-a-post-split}--\eqref{ineq:proof-a-post-conformity-error} conclude the reliability $\|\nabla_\pw^2 e\|_\Omega \lesssim \eta$.\\[-0.5em]

\noindent\emph{Efficiency.~} We start with the efficiency of $\mu$ from \eqref{def:mu} and adopt the notation $e_h \coloneqq I_h u - u_h$ from the proof of \Cref{thm:a-priori}. From \eqref{ineq:jump_R}--\eqref{ineq:normal_jump_R}, we deduce
\begin{align*}
	\sum_{F \in \F} (h_F^{-3}\|[R_h e_h]_F\|_F^2 + h_F^{-1}\|\partial_n[R_h e_h]_F\|_F^2)
	\lesssim |e_h|_{s_h}^2 + \|\nabla_\pw^2 R_h e_h\|^2_\Omega = \|e_h\|_h^2.
\end{align*}
This, the triangle, and trace inequality prove
\begin{align*}
	\mu^2 &\lesssim \sum_{F \in \F} (h_F^{-3}\|[G_h u]_F\|_F^2 + h_F^{-1}\|[\partial_n G_h u]_F\|_F^2) + \|e_h\|_h^2\nonumber\\
	&\lesssim \|h_\T^{-2}(u - G_h u)\|_\Omega^2 + \|h_\T^{-1}\nabla_\pw(u - G_h u)\|_\Omega^2 + \|\nabla_\pw^2(u - G_h u)\|^2_\Omega + \|e_h\|_h^2.
\end{align*}
In combination with \eqref{ineq:v-Gv-Poincare}, we obtain $\mu \lesssim \|e_h\|_h + \|\nabla_\pw^2(u - G_h u)\|_\Omega$.
The a~priori error estimate of \Cref{thm:a-priori} implies the efficiency 
\begin{align}\label{ineq:proof-efficiency-mu}
	\mu \lesssim \mathrm{osc}(f, \T) + \|\nabla_\pw^2(u - R_h u_h)\|_\Omega.
\end{align}
Since $|u_h|_{s_h} \lesssim \|\nabla_\pw^2 e\|_\Omega + \mathrm{osc}(f,\T)$ from \Cref{thm:a-priori}, $\eta \lesssim \|\nabla_\pw^2 e\|_\Omega + \mathrm{osc}(f,\T)$.
\end{proof}

\begin{remark}[bounds for conformity error]
    The constant hidden in the notation $\lesssim$ in \eqref{ineq:proof-a-post-conformity-error} critically depends on the polynomial degree $k$. Other strategies to bound the conformity error in the literature include the virtual enriching operator in \cite{BrennerSung2019}, where the dependency of the constant on the polynomial degree $k$ is currently unknown, and \cite{DongMascottoSutton2021} with explicit error bounds in $h$ and $k$ for dG discretizations.
\end{remark}

\begin{remark}[extension to polytopal meshes]\label{rem:polytopal-mesh}
    To extend the analysis of this paper to polytopal meshes, let us consider a partition $\mathcal{T}$ of $\Omega$ into closed polytopes with a given set $\mathcal{F}$ of faces, cf., e.g, \cite[Definition 1.4]{DiPietoDroniou2020} or \cite[Section 1.2.1]{CicuttinErnPignet2021} for a precise definition. Let $\mathcal{E}$ be a partition of $\partial \mathcal{F} \coloneqq \cup_{F \in \mathcal{F}} \partial F$ into closed $n-2$ dimensional polytopes with nonzero $n-2$ dimensional Lebesgue measure (called edges) such that $\cup \mathcal{E} = \partial \mathcal{F}$ and the intersection of any two distinct elements of $\mathcal{E}$ has vanishing $n-2$ dimensional Lebesgue measure. (In 2d, $n-2$ dimensional polytopes are points and, in 3d, they are straight lines.)

    For the mathematical analysis of the proposed numerical results, we assume that there exists a matching regular subtriangulation $\hat{\mathcal{T}}$ into simplices with the set $\hat{\mathcal{F}}$ of faces and the set $\hat{\mathcal{E}}$ of edges. Here, $\hat{\mathcal{T}}$ is a matching subtriangulation of $\mathcal{T}$ if
    \begin{enumerate}
        \item[(a)] for every simplex $\hat{T} \in \hat{\mathcal{T}}$, there exists a unique cell $T \in \mathcal{T}$ with $\hat{T} \subset T$,\\
        \item[(b)] for every face $\hat{F} \in \hat{\mathcal{F}}$ and $F \in \mathcal{F}$, either $\hat{F} \subset F$ or $\hat{F} \cap F$ has vanishing $n-1$ dimensional Lebesgue measure,\\
        \item[(c)] for every edge $\hat{E} \in \hat{\mathcal{E}}$ and $E \in \mathcal{E}$, either $\hat{E} \subset E$ or $\hat{E} \cap E$ has vanishing $n-2$ dimensional Lebesgue measure.
    \end{enumerate}
    The last condition is only meaningful for $n \geq 3$. If $n=2$, the last condition is replaced by either $\hat{E} = E$ or $\hat{E} \neq E$.
    This is a standard assumption so that, for instance, discrete trace and inverse inequalities hold and the constants therein depend on the shape regularity of the subtriangulation $\hat{\mathcal{T}}$, cf.~\cite[Definition 1.37]{DiPietroErn2015}, \cite[Definition 1.8]{DiPietoDroniou2020}, or \cite[Section 2.1.1]{CicuttinErnPignet2021}.
    The construction of a right-inverse $J_h$ of $I_h$ in \Cref{lem:right-inverse} on the simplicial submesh is straightforward. Apart from this, the remaining parts of \Cref{sec:error-analysis} carries over verbatim to polygonal meshes. This also applies to the analysis of the eigenvalue problem \eqref{def:pde-var} in \Cref{sec:eigenvalue_problem} below with the exception of LEB, where the convexity of cells is required for an explicit constant in the Poincar\'e inequality.
    
    Furthermore, we mention that employing a Lehrenfeld-Sch\"oberl typed stabilization in the case $\ell = k+2$ leads to the upper bound $s_h(I_h v, I_h v) \lesssim \|\nabla_\pw^2(v - \Pi_{\mathcal{T}}^{k+2} v)\|_{\Omega}$ for any $v \in H^2(\Omega)$ instead of \Cref{lem:best-appr-stab}. Thus, $\|\nabla_\pw^2(u - \Pi_{\mathcal{T}}^{k+2} u)\|_{\Omega}$ occurs as an additional error term on the right-hand side of the a~priori error estimate \eqref{ineq:a-priori}. On simplicial meshes,
    \begin{align*}
         \min_{p \in P_{k+2}(\mathcal{T})} \|\nabla^2_\pw(u - p)\|_\Omega \approx \|\nabla_\pw^2(u - \Pi_{\mathcal{T}}^{k+2} u)\|_{\Omega},
    \end{align*}
    cf.~\cite[Theorem 3.1]{CarstensenZhaiZhang2020} for the corresponding inequality in $H^1_0(\Omega)$. However, it is not known whether this holds on polytopal meshes as the local arguments of \cite{CarstensenZhaiZhang2020} lead to constants depending on the shape of the cells.
\end{remark}

\section{Biharmonic eigenvalue problem}\label{sec:eigenvalue_problem}
This section establishes a computational lower eigenvalue bound (LEB) for the biharmonic eigenvalue problem \eqref{def:pde-var} with (a~priori) optimal convergence rates and extends the analysis of \Cref{sec:error-analysis} to \eqref{def:pde-var}.

Recall that \eqref{def:pde-var} is a compact symmetric eigenvalue problem with positive eigenvalues
\begin{align*}
	0 < \lambda(1) < \lambda(2) \leq \lambda(3) \leq \dots
\end{align*}
and corresponding normalized eigenvectors $u(1), u(2), u(3), \dots$ with 
\begin{align*}
	(u(j), u(k))_\Omega = \delta_{jk} \quad\text{for any } j,k \in \mathbb{N},
\end{align*}
where $\delta_{jk}$ is the Kronecker delta.

\subsection{Discrete eigenvalue problem}
Higher-order computational LEB require a fine-tuning of the stabilization and explicit constants in \Cref{lem:best-appr-stab}. For the latter purpose, we introduce the following weights. For a simplex $T \in \T$ with a face $F$ and an edge $E$, we define
\begin{align*}
    \ell_T(F) \coloneqq h_T^2|F|/|T| \quad\text{and}\quad \ell_T(E,F) \coloneqq \begin{cases}
        h_T^2|E|/|F| &\mbox{if } n = 3,\\
        h_T^2/|F| &\mbox{if } n = 2.
    \end{cases} 
\end{align*}
Given a simplex $T \in \T$ and a positive parameter $\sigma > 0$, we replace the local stabilization $s_T$ in \eqref{def:stab_loc} by
\begin{align}\label{def:stab-eig}
	&s_T(u_h,v_h) \coloneqq \sigma h_T^{-4} (\Pi_T^{\ell}(u_{T} - R_T u_h), \Pi_{T}^{\ell}(v_T - R_T v_h))_{T}\nonumber\\
	&\quad + \sigma h_T^{-2} \sum_{F \in \mathcal{F}(T)} \ell_T^{-1}(F)(\Pi_{F}^{m}(u_{\F(T)} - R_T u_h),\Pi_{F}^{m}(v_{\F(T)} - R_T v_h))_{F}\nonumber\\
	&\quad +\sigma\sum_{F\in\mathcal{F}(T)} \ell_T^{-1}(F)(\Pi_{F}^k(\alpha_{\F(T)} - \partial_n R_T u_h),\Pi_{F}^k(\beta_{\F(T)}-\partial_n R_T v_h))_{F}\\
	&\quad + \sigma \sum_{F \in \mathcal{F}(T)}\ell_T^{-1}(F)\sum_{E \in \E(F)} \ell_T^{-1}(E,F)(\Pi_{E}^{k}(u_{\E(T)} - R_T u_h), \Pi_{E}^{k}(v_{\E(T)} - R_T v_h))_{E}.\nonumber
\end{align}
We mention that the weights $\ell(F,T)$ and $\ell(E,F)$ are, up to multiplicative constants depending on the shape of $T$, equivalent to $h_T$ for any $F \in \F(T)$ and $E \in \E(F)$. Thus, this stabilization is equivalent to $s_T$ of \eqref{def:stab_loc}.
Apart from this modification for the stabilization, we adopt the remaining notation from \Cref{sec:discretization}.
Let $\ell \coloneqq k+2$ so that the discrete ansatz space reads $V_h \coloneqq P_{k+2}(\T) \times P_m(\F(\Omega)) \times P_k(\F(\Omega)) \times P_k(\E(\Omega))$.
The right-hand side of \eqref{def:pde-var} is approximated by the bilinear form
\begin{align*}
	b_h(u_h,v_h) &\coloneqq (u_\T,v_\T)_\Omega
\end{align*}
for $u_h = (u_\T, u_\F, \alpha_\F, u_\E), v_h = (v_\T, v_\F, \beta_\F, v_\E) \in V_h$, which induces the seminorm $|\bullet|_{b_h} \coloneqq b_h(\bullet,\bullet)$.
The discrete eigenvalue problem seeks the eigenpairs $(\lambda_h, u_h) \in \mathbb{R} \times V_h$ such that
\begin{align}\label{def:discrete-eig-problem}
	a_h(u_h,v_h) = \lambda_h b_h(u_h,v_h) \quad\text{for any } v_h \in V_h,
\end{align}
where $s_h(u_h, v_h) \coloneqq \sum_{T \in \T} s_T(u_h|_T, v_h|_T)$ with $s_T$ from 
\eqref{def:stab-eig}. On the discrete level,
the eigenvalue problem \eqref{def:discrete-eig-problem} leads to $N \coloneqq \dim(P_{k+2}(\T)
)$ positive eigenvalues
\begin{align*}
	0 < \lambda_h(1) \leq \dots \leq \lambda_h(N)
\end{align*}
and the associated discrete eigenvectors $u_h(1), \dots, u_h(N)$ with
\begin{align*}
	b_h(u_h(j),u_h(k)) = \delta_{jk}.
\end{align*}

\subsection{Lower eigenvalue bounds}\label{sec:LEB}
The HHO eigensolver of this paper satisfies the framework of \cite{CarstensenZhaiZhang2020,Tran2024x2} and allows for the following LEB.
\begin{theorem}[LEB]\label{thm:LEB}
    For any $1 \leq j \leq N$,
    \begin{equation*}
        \mathrm{LEB}(j) \coloneqq \min\{1,1/(\alpha +\beta\lambda_h(j))\}\lambda_h(j) \leq\lambda(j)
    \end{equation*}
    with the constants $\alpha \coloneqq \sigma(1/\pi^4 + c_\mathrm{tr}/\pi^2 + c_\mathrm{tr} + c_\mathrm{tr}(2/\pi + n/\pi^2))$, $\beta \coloneqq h^4/\pi^4$,
    and $c_\mathrm{tr} \coloneqq (2+(n+1)/\pi)/\pi$, where $h = \|h_\T\|_{L^\infty(\Omega)}$ is the maximum mesh size in $\T$.
\end{theorem}
\begin{remark}[quasi optimality]\label{rem:direct-LEB}
    Since $\beta \to 0$ as $h \to 0$, the choice $\sigma < (1/\pi^4 + c_\mathrm{tr}/\pi^2 + c_\mathrm{tr} + c_\mathrm{tr}(2/\pi + n/\pi^2))^{-1}$ leads to $\mathrm{LEB}(j) = \lambda_h(j)$ for sufficiently small mesh sizes $h$. The a~priori error estimate of \Cref{thm:a-priori-eig} proves the quasi-best approximation $\lambda(j) - \mathrm{LEB}(j) \lesssim \|\nabla_\pw^2(u(j) - G_h u(j))\|_\Omega^2$.
\end{remark}
The proof of \Cref{thm:LEB} utilizes the following explicit constant in \Cref{lem:best-appr-stab}.
\begin{lemma}[optimality of $s_T$]\label{lem:best-appr-stab-eig}
    Let a simplex $T$ be given. Any $v \in H^2(T)$ satisfies
    \begin{align*}
        s_T(I_T v, I_T v) \leq (1/\pi^4 + c_\mathrm{tr}/\pi^2 + c_\mathrm{tr} + c_\mathrm{tr}(2/\pi + n/\pi^2)) \|\nabla^2(v - G_h v)\|_T^2.
    \end{align*}
\end{lemma}
\begin{proof}
    The function $r \coloneqq v - G_h v$ satisfies the vanishing mean properties \eqref{eq:vanishing_mean}. Hence, Proposition 4.3 of \cite{CarstensenZhaiZhang2020} and the Poincar\'e inequality prove
    \begin{align}\label{ineq:proof-stab-const}
        h_T^{-4} \|r\|^2_T + \sum_{F \in \mathcal{F}(T)} (h_T^{-2} \ell_T^{-1}(F) \|r\|_F^2 + \ell_T^{-1}(F) \|\partial_n r\|_F^2)&\nonumber\\
        \leq (1/\pi^4 + c_\mathrm{tr}/\pi^2 + c_\mathrm{tr})\|\nabla^2 r\|_T^2&.
    \end{align}
    From \cite[Ineq.~(4.1)]{CarstensenZhaiZhang2020} (applied to the face $F$), we infer, for any $F \in \F(T)$, that
    \begin{align*}
        \sum_{E \in \E(F)} \ell_T^{-1}(E,F) \|r\|_E^2 \leq \frac{n}{h_T^2}\|r\|^2_F + \frac{2}{h_T}\|r\|_F\|\partial_t r\|_F.
    \end{align*}
    The sum of this over all sides $F \in \F(T)$, the Cauchy inequality, and \cite[Proposition 4.3]{CarstensenZhaiZhang2020} provide
    \begin{align*}
        &\sum_{F \in \mathcal{F}(T)}\ell_T^{-1}(F)\sum_{E \in \E(F)} \ell_T^{-1}(E,F) \|r\|_E^2 \leq \frac{n}{h_T^2}\sum_{F \in \F(T)} \ell_T^{-1}(F) \|r\|_F^2\\
        &\qquad + \frac{2}{h_T}\big(\sum_{F \in \F(T)} \ell_T^{-1}(F) \|r\|_F^2\big)^{1/2}\big(\sum_{F \in \F(T)} \ell_T^{-1}(F) \|\partial_t r\|_F^2\big)^{1/2}\\
        & \leq \big(\frac{n c_\mathrm{tr}}{h_T^2}\|\nabla r\|_T^2 + \frac{2c_\mathrm{tr}}{h_T}\|\nabla r\|_T\|\nabla^2 r\|_T\big).
    \end{align*}
    This, \eqref{ineq:proof-stab-const}, and \eqref{ineq:v-Gv-Poincare} conclude the assertion.
\end{proof}

\begin{proof}[Proof of \Cref{thm:LEB}]
    Three conditions involving the Galerkin projection $G_h = R_h \circ I_h : V \to P_{k+2}(\T)$ in \cite{Tran2024x2} give rise to LEB.
    From \Cref{lem:commuting}, we deduce the first condition
    \begin{align*}
        \|\nabla_\pw^2 R_h I_h v\|^2_\Omega = \|\nabla^2 v\|_\Omega^2 - \|\nabla_\pw^2(v - G_h v)\|^2_\Omega
    \end{align*}
    for any given $v \in V$.
    The second condition is the optimality of $s_T$ in \Cref{lem:best-appr-stab-eig}.
    Finally, \eqref{ineq:v-Piv} establishes the third condition
    \begin{align*}
        |I_h v|_{b_h}^2 &= \|\Pi_\T^{k+2} v\|_\Omega^2 = \|v\|_\Omega^2 - \|(1 - \Pi_\T^{k+2})v\|_\Omega^2\\ &\geq \|v\|_\Omega^2 - h^4\|\nabla_\pw^2(1 - G_h) v\|_\Omega^2/\pi^4.
    \end{align*}
    Theorem 3.3 and Remark 3.6 from \cite{Tran2024x2} concludes the assertion.
\end{proof}

\subsection{Spectral correctness}
Before error estimates are presented, we establish the spectral correctness of the HHO method following the arguments of \cite{CaloCicuttinDengErn2019} in the framework of Babu\v{s}ka-Osborn theory \cite{BabuskaOsborn1991}. Since the application of \cite{CaloCicuttinDengErn2019} is rather straight-forward, we only briefly outline the main arguments.
The eigenpairs of \eqref{def:pde-var} can be described by the eigenpairs of the operator $T : L^2(\Omega) \to L^2(\Omega)$ defined, for any $\phi \in L^2(\Omega)$, by
\begin{align*}
	(\nabla^2 T \phi, \nabla^2 \psi)_\Omega = (\phi, \psi)_\Omega \quad\text{for any } \psi \in V.
\end{align*}
In fact, $(\lambda,u)$ is an eigenpair of \eqref{def:pde-var} if and only if $(\lambda^{-1},u)$ is an eigenpair of $T$.
On the discrete level, the solution operator $T_h : L^2(\Omega) \to V_h$ associated with \eqref{def:discrete-eig-problem} maps $\phi \in L^2(\Omega)$ onto the unique solution $T_h \phi \in V_h$ to
\begin{align}\label{def:Th}
	a_h(T_h \phi, v_h) = (\phi, v_\T)_\Omega \quad\text{for any } v_h = (v_\T, v_\F, \beta_{\F}, v_\E) \in V_h.
\end{align}
To construct an appropriate operator $\hat{T}_h : L^2(\Omega) \to L^2(\Omega)$ for the setting of \cite{BabuskaOsborn1991},
let $V_\T \coloneqq P_\ell(\T)$ and $V_\Sigma \coloneqq P_m(\F(\Omega)) \times P_k(\F(\Omega)) \times P_k(\E(\Omega))$ so that $V_h = V_\T \times V_\Sigma$.
We consider the operator $Z_h : V_\T \to V_\Sigma$ defined, for any $v_\T \in V_\T$, by
\begin{align*}
	a_h((0, Z_h v_\T), (0, w_\Sigma)) = a_h((v_\T,0),(0,w_\Sigma)) \quad\text{for all } w_\Sigma \in V_\Sigma.
\end{align*}
Since $\|(0,\bullet)\|_h$ is a norm in $V_\Sigma$ from \Cref{thm:wellposedness}, this defines $Z_h v_\T$ uniquely.
We define the bilinear form
\begin{align*}
	a_\T(u_\T, v_\T) \coloneqq a_h((u_\T, Z_h u_\T), (v_\T, Z_h v_\T)) \quad\text{for any } u_\T, v_\T \in V_\T
\end{align*}
and introduce the solution operator $\hat{T}_h : L^2(\Omega) \to V_\T \subset L^2(\Omega)$ with
\begin{align*}
	a_\T(\hat{T}_h \phi, v_\T) = (\phi, v_\T)_\Omega \quad\text{for any } \phi \in L^2(\Omega) \text{ and } v_\T \in V_\T.
\end{align*}
The face and edge variables are eliminated by the operator $Z_h$ and so, 
\begin{align}\label{eq:T-hatT-relation}
	T_h \phi = (\hat{T}_h \phi, Z_h \hat{T}_h \phi) \quad\text{for any } \phi \in L^2(\Omega),
\end{align}
cf.~\cite[Lemma 3.1]{CaloCicuttinDengErn2019} for a precise proof of this. Thus, $\lambda_h$ is a discrete eigenvalue of \eqref{def:discrete-eig-problem} if and only if $\lambda_h^{-1}$ is an eigenvalue of $\hat{T}_h$.
The convergence of $\hat{T}_h$ to $T$ can be obtained from the error estimates of \Cref{sec:error-analysis}.
Recall the index $s$ of elliptic regularity from \eqref{ineq:elliptic-reg}.
\begin{theorem}[Convergence of solution operators]\label{thm:spectral-correctness}
	It holds
	\begin{align*}
		\|T - \hat{T}_h\|_{\mathcal{L}(L^2(\Omega); L^2(\Omega))} \lesssim h^s.
	\end{align*}
	In particular, for any eigenvalue $\mu$ of $T$ with algebraic multiplicity $r$, there exist $r$ eigenvalues $\mu_{h,1}, \dots, \mu_{h,r}$ of $\hat{T}_h$ with $\max_{1 \leq j \leq r}\{\mu - \mu_{h,r}\} \lesssim h^s$.
\end{theorem}
\begin{proof}
	Given any $L^2$ functions $\phi, \psi \in L^2(\Omega)$, \eqref{def:Th}--\eqref{eq:T-hatT-relation} imply
	\begin{align*}
		(T\phi - \hat{T}_h\phi, \psi)_\Omega &= ((1 - \Pi_\T^\ell) T\phi, \psi)_\Omega + (\Pi_\T^\ell T\phi - \hat{T}_h \phi, \psi)_\Omega\\
		&= ((1 - \Pi_\T^\ell) T\phi, \psi)_\Omega + a_h(I_h T \phi - T_h \phi, T_h \psi)\\
		&\leq \|(1 - \Pi_\T^\ell) T\phi\|_\Omega\|\psi\|_\Omega + \|I_h T \phi - T_h \phi\|_h\|T_h \psi\|_h.
	\end{align*}
	\Cref{thm:a-priori} and the elliptic regularity of $T \phi \in H^{2+s}(\Omega)$ provide $\|I_h T \phi - T_h \phi\|_h \lesssim \mathrm{osc}(\phi, \T)+ \min_{p \in P_{k+2}(\T)} \|\nabla_\pw^2(T \phi - p)\|_\Omega \lesssim h^s\|\phi\|_\Omega$. Therefore,
	\begin{align*}
		\|(1 - \Pi_\T^\ell) T\phi\|_\Omega + \|I_h T \phi - T_h \phi\|_h \lesssim h^2\|\nabla^2 T\phi\|_\Omega + h^s\|\phi\|_\Omega \lesssim h^s\|\phi\|_\Omega
	\end{align*}
	from a Poincar\'e inequality and the stability $\|\nabla^2 T \phi\|_\Omega \lesssim \|\phi\|_\Omega$ of the solution operator $T$.
	A triangle inequality and \Cref{thm:a-priori} prove
	\begin{align*}
		\|\T_h \psi\|_h \leq \|T_h \psi - I_h T \psi\|_h + \|I_h T \psi\|_h \lesssim \mathrm{osc}(\psi, \T) + \|\nabla^2 T \psi\|_\Omega \lesssim \|\psi\|_{\Omega}.
	\end{align*}
	The combination of the three previously displayed formula concludes the proof.
\end{proof}
We mention that the convergence rates given in \Cref{thm:spectral-correctness} are suboptimal, but can be recovered using the arguments of \cite{CaloCicuttinDengErn2019}. However, the sole goal was to establish the spectral correctness of the method as a detailed error analysis (that lead to optimal convergence rates) is presented below.
Furthermore, the spectral correctness also applies to any $\ell \geq \max\{k-2,1\}$ (not just $\ell = k+2$).

\subsection{A~priori}
In the remaining parts of this section, we extend the error analysis of the source problem in \Cref{sec:error-analysis} to the eigenvalue problem \eqref{def:pde-var} using the arguments of \cite{CarstensenGallistlSchedensack2015,CarstensenGraessleTran2024}. For the sake of brevity, the analysis is only carried out for simple eigenvalues. However, we mention that an extension to eigenvalue clusters is possible following \cite{DaiHeZhou2015,Gallistl2015}. 

Let $\lambda = \lambda(j)$ be a simple eigenvalue with the associated eigenvector $u = u(j)$. For $1 \leq j \leq \dim(P_{\ell}(\T))$ (with $\ell = k+2$), the discrete problem \eqref{def:discrete-eig-problem} provides an approximation $\lambda_h = \lambda_h(j)$ of $\lambda$ with the discrete eigenvector $u_h = (u_\T, u_\F, \alpha_\F, u_\E) = u_h(j) \in V_h$.
The sign of the eigenvectors can be chosen so that $(u, u_\T)_\Omega \geq 0$.
Recall the index $2\geq s > 0$ of elliptic regularity from \eqref{ineq:elliptic-reg} and $t \coloneqq \min\{s,\ell-1\}$ from \Cref{thm:L2}.

\begin{theorem}[a~priori]\label{thm:a-priori-eig}
	Assume that the mesh size $h$ is sufficiently small. Then
	\begin{align*}
		|\lambda - \lambda_h| + \|I_h u - u_h\|_{h}^2 + h^{-2t} \| u - u_\T\|_\Omega^2\lesssim \min_{p \in P_{k+2}(\mathcal{T}_h)}\|\nabla_\pw^2(u - p)\|_\Omega^2.
	\end{align*}
\end{theorem}

\begin{proof}
	The proof departs from the auxiliary problem that seeks the discrete solution $\widetilde{u}_h = (\widetilde{u}_\T, \widetilde{u}_\F, \widetilde{\alpha}_\F, \widetilde{u}_\E) \in V_h$ to \eqref{def:discrete_problem} with the right-hand side $f = \lambda u$, i.e.,
	\begin{align*}
		a_h(\widetilde{u}_h, v_h) = \lambda(u, v_\T)_\Omega\quad\text{for any } v_h = (v_\T, v_\F, \beta_\F, v_\E) \in V_h.
	\end{align*}
	For sufficiently small mesh size $h$, \Cref{thm:spectral-correctness} shows that 
    $$C \coloneqq \sqrt{2}(1 + \max_{k \in \{1,\cdots,N\} \setminus \{j\}}|\lambda /(\lambda-\lambda_h(k))|) \lesssim 1$$
    is uniformly bounded. It is known from \cite[Lemma 2.4]{CarstensenGallistlSchedensack2015} that
	$\| u - u_\T\|_\Omega \leq C \|u - \widetilde{u}_\T\|_\Omega$.
	Hence, a triangle inequality, the $L^2$ estimate of \Cref{thm:L2}, and \eqref{eq:best-appr}--\eqref{ineq:v-Gv-Poincare} imply
	\begin{align*}
		&\|u - u_\T\|_\Omega \lesssim \|u - \widetilde{u}_\T\|_\Omega \leq \|(1 - \Pi_\T^\ell) u\|_\Omega + \|\Pi_\T^\ell u - \widetilde u_\T\|_\Omega\\
		&\leq \|(1 - G_h) u\|_\Omega + \|\Pi_\T^{\ell} u - \widetilde u_\T\|_\Omega \lesssim h^{t}\min_{p \in P_{k+2}(\T)} \|\nabla_\pw^2(u - p)\|_\Omega + h^{t}\mathrm{osc}(\lambda u, \T).
	\end{align*}
	Notice that, for $\ell = k+2$, $\mathrm{osc}(\lambda u, \T) \leq \|h_\T^2\lambda(u - G_h u)\|_\Omega \lesssim h^{4} \|\nabla_\pw^2(u - G_h u_h)\|_\Omega$ from \eqref{ineq:v-Gv-Poincare} and with \eqref{eq:best-appr},
	\begin{align}\label{ineq:L2-eig}
		\|u - u_\T\|_\Omega \lesssim h^{t}\min_{p \in P_{k+2}(\T)} \|\nabla_\pw^2(u - p)\|_\Omega.
	\end{align}
	With the abbreviation $e_h = I_h u - u_h$, it remains to establish
	\begin{align*}
		|\lambda - \lambda_h| + \|e_h\|_{h}^2 \lesssim \min_{p \in P_{k+2}(\T)} \|\nabla_\pw^2(u - p)\|_\Omega^2.
	\end{align*}
	Elementary algebra with the normalization $\|u\|_\Omega = 1= \|u_\T\|_\Omega$ leads to $2\lambda = \lambda\|u - u_\T\|_\Omega^2 + 2\lambda(u,u_\T)_{\Omega}$. This and $\|e_h\|_{h}^2 - \lambda_h = \|I_h u\|_{h}^2 - 2a_h(I_h u, u_h)$ from $\lambda_h = \|u_h\|_h^2$ show the split 
	\begin{align}\label{eq:proof-a-pr-eig-split}
		\lambda - \lambda_h + \|e_h\|_{h}^2 = \lambda\|u - u_\T\|_\Omega^2 + \|I_h u\|_{h}^2 - \lambda + 2(\lambda(u,u_\T)_{\Omega} - a_h(I_h u, u_h)).
	\end{align}
	The individual terms on the right-hand side are bounded as follows.
	\Cref{lem:commuting} proves $\lambda = \|\nabla^2 u\|_\Omega^2 \geq \|\nabla_\pw^2 G_h u\|_\Omega^2$ and so,
	\begin{align}\label{ineq:proof-a-pr-eig-1}
		\|I_h u\|_{h}^2 - \lambda = \|\nabla_\pw^2 G_h u\|_\Omega^2 + |I_h u|_{s_h}^2 - \|\nabla^2 u\|_\Omega^2 \leq |I_h u|_{s_h}^2.
	\end{align}
	The variational formulation \eqref{def:pde-var} proves
	\begin{align*}
		\lambda(u,u_\T)_{\Omega} &= \lambda(u, J_h u_h)_{\Omega} - \lambda((1 - \Pi_\T^{\ell})u, J_h u_h)_{\Omega}\\
		&= (\nabla^2 u, \nabla^2 J_h u_h)_{\Omega} - \lambda((1 - \Pi_\T^{\ell})u, J_h u_h)_{\Omega}.
	\end{align*}
	This and the identity $a_h(I_h u, u_h) = (\nabla_\pw^2 G_h u, \nabla^2 J_h u_h)_\Omega + s_h(I_hu,u_h)$ from \Cref{lem:commuting} provide
	\begin{align*}
		\lambda(u,u_\T)_{\Omega} - a_h(I_h u, u_h) &= (\nabla^2_\pw(u - G_h u), \nabla^2 J_h u_h)_\Omega\\
		&\quad - \lambda((1 - \Pi_\T^\ell)u, J_h u_h)_{\Omega} - s_h(I_h u, u_h).
	\end{align*}
	From \eqref{ineq:v-Gv-Poincare}--\eqref{ineq:v-Piv} and the orthogonality \eqref{eq:commuting-G}, we deduce 
	\begin{align*}
		(\nabla^2_\pw(u - G_h u), \nabla^2 J_h u_h)_\Omega
		- \lambda((1 - \Pi_\T^\ell)u, J_h u_h)_{\Omega}&\\
		\lesssim  \|\nabla^2_\pw(u - G_h u)\|_\Omega \|\nabla^2_\pw (J_h u_h - R_h u_h)\|_\Omega&.
	\end{align*}
	The combination of the two previously displayed formula results in
	\begin{align}\label{ineq:proof-a-priori-eig}
		\begin{aligned}
			&\lambda(u,u_\T)_{\Omega} - a_h(I_h u, u_h)\\
			&\qquad \lesssim \|\nabla^2_\pw(u - G_h u)\|_\Omega\|\nabla_\pw^2(J_h u_h - R_h u_h)\|_\Omega - s_h(I_h u, u_h).
		\end{aligned}
	\end{align}
	Recall the estimate $\|\nabla_\pw^2(J_h - R_h) I_h u\|_\Omega \lesssim \|\nabla_\pw^2(u - G_h u)\|_\Omega$ from \eqref{ineq:quasi-best-J-R}.
	The triangle inequality and continuity of $J_h$ in \Cref{lem:right-inverse} imply
	\begin{align*}
		\|\nabla_\pw^2(J_h u_h - R_h u_h)\|_\Omega &\leq \|\nabla^2 J_h e_h\|_\Omega + \|\nabla_\pw^2 R_h e_h\|_\Omega + \|\nabla_\pw^2 (J_h - R_h) I_h u\|_\Omega\\
		& \lesssim \|e_h\|_h + \|\nabla_\pw^2(u - G_h u)\|_\Omega.
	\end{align*}
	This, \eqref{ineq:proof-a-priori-eig}, $-s_h(I_h u, u_h) \leq |I_h u|^2_{s_h} + |I_h u|_{s_h}|e_h|_{s_h}$, and \Cref{lem:best-appr-stab} provide
	\begin{align*}
		\lambda(u,u_\T)_{\Omega} - a_h(I_h u, u_h) \lesssim \|\nabla^2_\pw(u - G_h u)\|_\Omega(\|e_h\|_h + \|\nabla_\pw^2(u - G_h u)\|_\Omega).
	\end{align*}
	In combination with \eqref{ineq:L2-eig}--\eqref{ineq:proof-a-pr-eig-1} and \Cref{lem:best-appr-stab}, we conclude
	\begin{align*}
		\lambda - \lambda_h + \|e_h\|_h^2 \lesssim \|\nabla_\pw^2(u - G_h u)\|_\Omega^2.
	\end{align*}
	If $\sigma$ is chosen sufficiently small so that $\lambda_h \leq \lambda$ for sufficiently small mesh sizes $h$ from \Cref{rem:direct-LEB}, then the proof is complete. Otherwise, we need to bound of $\lambda_h - \lambda + \|e_h\|_h^2$. Elementary algebra with $\|u\|_\Omega = 1 = \|u_\T\|_\Omega$ and \eqref{def:discrete-eig-problem} lead to $2 \lambda_h = \lambda_h\|u - u_\T\|^2_\Omega + 2\lambda_h(u, u_\T)_\Omega = \lambda_h\|u - u_\T\|^2_\Omega + 2a_h(u_h, I_h u)$. This, $\lambda = \|\nabla^2 u\|_\Omega^2 \geq \|\nabla_\pw^2 G_h u\|_\Omega^2$ from \eqref{eq:commuting-G}, and $\lambda_h = \|u_h\|_h^2$ yield
	\begin{align*}
		\lambda_h - \lambda + \|e_h\|_h^2 &= 
		\lambda_h\|u - u_\T\|^2_\Omega - \lambda - \lambda_h + 2a_h(u_h, I_h u) + \|e_h\|_h^2\\
		&\leq
		\lambda_h\|u - u_\T\|^2_\Omega - \|\nabla_\pw^2 G_h u\|_\Omega^2 - \|u_h\|^2_h + 2a_h(u_h, I_h u) + \|e_h\|_h^2\\
		&\leq \lambda_h\|u - u_\T\|^2_\Omega + |I_h u|_{s_h}^2,
	\end{align*}
	where $\|e_h\|_h^2 = \|u_h\|_h^2 + \|I_h u\|^2_h - 2a_h(u_h, I_h u)$ is utilized in the last step.
	The assertion $\lambda_h - \lambda + \|e_h\|_h^2 \lesssim \|\nabla_\pw^2(u - G_h u)\|^2_\Omega$ follows from \eqref{ineq:L2-eig} and \Cref{lem:best-appr-stab}.
\end{proof}

\subsection{A~posteriori}
The error control of \Cref{thm:a-posteriori} motivates the error estimator
\begin{align*}
	\widehat{\eta}^2 \coloneqq \mu^2 + |u_h|_{s_h}^2
\end{align*}
for the biharmonic eigenvalue problem, where $\mu$ is defined in \eqref{def:mu}.
\begin{theorem}[a~posteriori]\label{thm:a-posteriori-eig}
	For sufficiently small mesh sizes $h$, it holds
	\begin{align*}
		\|\nabla_\pw^2(u - R_h u_h)\|_\Omega \approx \widehat{\eta}.
	\end{align*}
\end{theorem}
\begin{proof}
	The proof departs from the auxiliary problem that seeks $\widetilde{u} \in V$ such that
	\begin{align*}
		(\nabla^2 \widetilde{u}, \nabla^2 v)_\Omega = \lambda_h(u_\T, v)_\Omega \quad\text{for any } v \in V.
	\end{align*}
	In other words, $\widetilde{u}$ is the unique solution to the biharmonic problem with right-hand side $\lambda_h u_\T$ and so, the a posteriori error control from \Cref{thm:a-posteriori} provides
	\begin{align*}
		\|\nabla^2_\pw(\widetilde{u} - R_h u_h)\|_\Omega \lesssim \widehat{\eta},
	\end{align*}
	where the data oscillation $\mathrm{osc}(\lambda_h u_\T, \T) = 0$ vanishes.
	The error between $u$ and $\widetilde{u}$ is bounded by the error of the right-hand side
	\begin{align}\label{ineq:stability-eig}
		\|\nabla^2(u - \widetilde{u})\|_\Omega \lesssim \|\lambda u - \lambda_h u_\T\|_\Omega.
	\end{align}
	Hence, the triangle inequality implies
	\begin{align}\label{ineq:proof-a-post-eig}
		\|\nabla^2_\pw(u - R_h u_h)\|_\Omega \lesssim \widehat{\eta} + \|\lambda u - \lambda_h u_\T\|_\Omega.
	\end{align}
	The inequality (6.7) of \cite{CarstensenGraessleTran2024} applies and provides
	\begin{align}\label{ineq:proof-eig-hot}
		\|\lambda u - \lambda_h u_\T\|_\Omega \lesssim h^s\|\nabla_\pw^2(u - G_h u)\|_\Omega
	\end{align}
	from \eqref{ineq:L2-eig}.
	Hence, for sufficient small mesh sizes $h$, the reliability $\|\nabla^2_\pw(u - R_h u_h)\|_\Omega \lesssim \widehat{\eta}$ follows from \eqref{ineq:proof-a-post-eig}.
	Recall that $\widehat{\eta}$ is efficient in the sense
	\begin{align*}
		\widehat{\eta} \lesssim \|\nabla_\pw^2(\widetilde{u} - R_h u_h)\|_\Omega.
	\end{align*}
	This, a triangle inequality, \eqref{ineq:stability-eig}, and \eqref{ineq:proof-eig-hot} conclude the efficiency
	\begin{align*}
		\widehat{\eta} &\lesssim \|\nabla_\pw^2(u - R_h u_h)\|_\Omega + \|\nabla^2(u - \widetilde{u})\|_\Omega\\
		& \lesssim \|\nabla_\pw^2(u - R_h u_h)\|_\Omega + \|\lambda u - \lambda_h u_\T\|_\Omega \lesssim \|\nabla_\pw^2(u - R_h u_h)\|_\Omega.\qedhere
	\end{align*}
\end{proof}

\section{Numerical examples}\label{sec:num_examples}
This section presents 2D benchmarks for the source and eigenvalue problem in the L-shaped domain $\Omega \coloneqq (-1,1)^2 \setminus [0,1]\times[-1,0]$ with a initial triangulation consisting of six triangles.

\subsection{Biharmonic problem}\label{sec:num_ex_bihar}
We set $\ell = k+2$ and note that the volume variables can be statically condensed.

\begin{figure}[ht]
    \centering
    \includegraphics[width=0.6\linewidth]{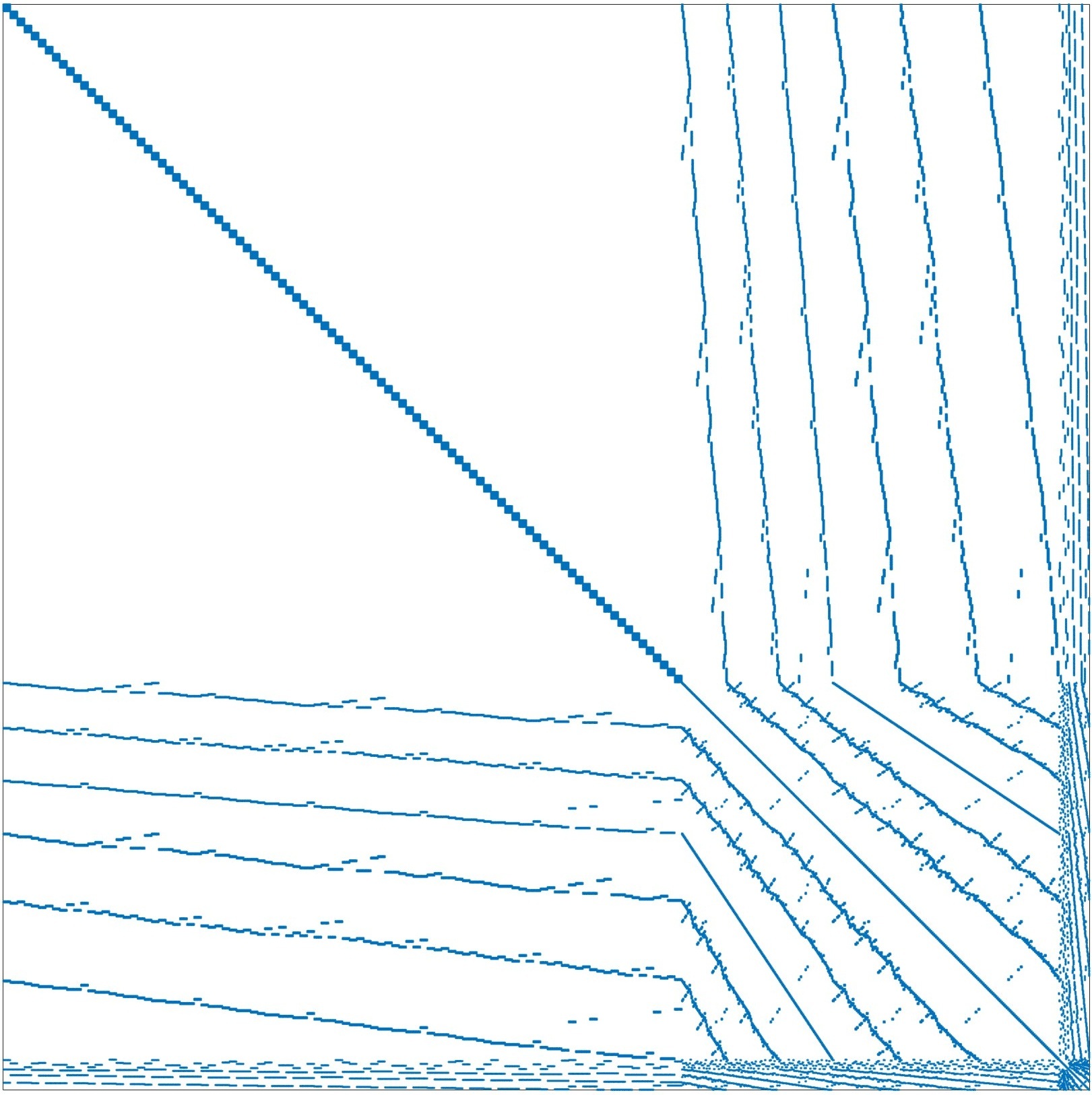}
    \caption{Sparsity pattern of stiffness matrix for $k = 2$ on uniform mesh with 2305 degrees of freedom}
    \label{fig:sparsity-pattern}
\end{figure}

\begin{figure}[ht]
	\begin{minipage}{0.68\textwidth}
		\centering
		\includegraphics[height=7cm]{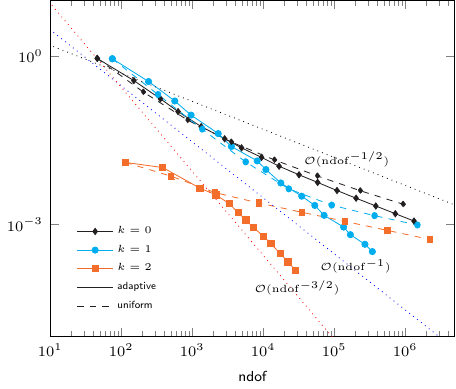}
	\end{minipage}\hfill
	\begin{minipage}{0.3\textwidth}
		\centering
		\includegraphics[height=3.5cm]{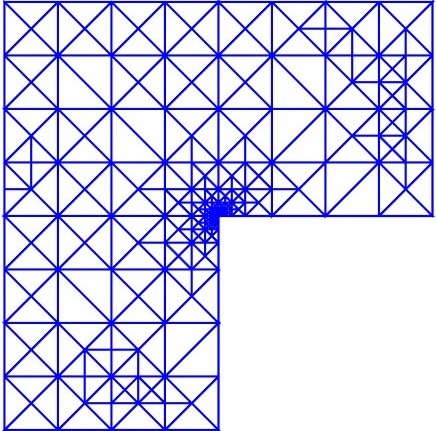}
	\end{minipage}
	\captionsetup{width=1\linewidth}
	\caption{Convergence history plot of $\eta$ (left) and adaptive triangulation with 445 triangles (right) for the experiment in \Cref{sec:num_ex_bihar}}
	\label{fig:biharm}
\end{figure}

\subsubsection{Adaptive algorithm}
The a~posteriori error control from \Cref{thm:a-posteriori} motivates the local refinement indicator
\begin{align*}
	\eta^2(T) \coloneqq \sum_{F \in \F(T)} (h_F^{-3} \|[R_h u_h]_F\|_F^2 + h_F^{-1} \|[\partial_n R_h u_h]_F\|_F^2)&\\
	+ s_T(u_h,u_h) + h_T^4\|(1 - \Pi_T^\ell) f\|_T^2&,
\end{align*}
which is utilized in a standard adaptive loop with the D\"orfler marking strategy and bulk parameter $\vartheta = 1/2$ \cite{Doerfler1996}, i.e., at each refinement step, we mark a subset $\mathcal{M} \subset \T$ of minimal cardinality such that
\begin{align*}
	\eta^2 \leq \frac{1}{2}\sum_{T \in \mathcal{M}} \eta^2(T).
\end{align*}
The convergence history plots display the quantities of interest against the number of degrees of freedom $\mathrm{ndof}$ in a log-log plot for uniform and adaptive computations. (Recall the scaling $\mathrm{ndof} \approx h^2$ for uniform meshes.)

\subsubsection{Numerical benchmark}
We approximate the unknown solution $u$ to the source problem \eqref{def:source-problem} with the constant right-hand side $f \equiv 1$ in $\Omega$, which leads to a vanishing data oscillation $\mathrm{osc}(f,\T) = 0$.

In \Cref{fig:sparsity-pattern}, the sparsity pattern of the stiffness matrix for $k = 2$ on a uniform mesh is displayed. The degrees of freedom are ordered as follows: First, we have the degrees of freedom associated with the mesh elements, followed by edge degrees of freedom, and finally one degree of freedom for each vertex of the triangulation. It can be observed that the upper left part of the stiffness matrix can be statically condensed by Schur complement. Furthermore, \Cref{fig:sparsity-pattern} shows that the lower part of the stiffness matrix is more dense. This is expected since vertex degrees of freedom commute with the whole vertex patch. However, there are significantly less vertex than edge degrees of freedom.

The convergence history of the error estimator $\eta$ from \Cref{thm:a-posteriori} is displayed in \Cref{fig:biharm} (left) with suboptimal convergence rates on uniformly refined meshes. The adaptive algorithm refines towards the expected singularity at the origin as displayed in \Cref{fig:biharm} (right) and recovers the optimal convergence rates $(k+1)/2$ for all displayed polynomial degrees $k = 0, 1, 2$.

\begin{figure}[ht]
	\begin{minipage}{0.68\textwidth}
		\centering
		\includegraphics[height=7cm]{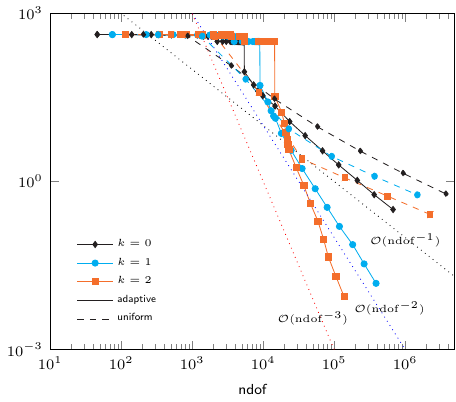}
	\end{minipage}\hfill
	\begin{minipage}{0.3\textwidth}
		\centering
		\includegraphics[height=3.5cm]{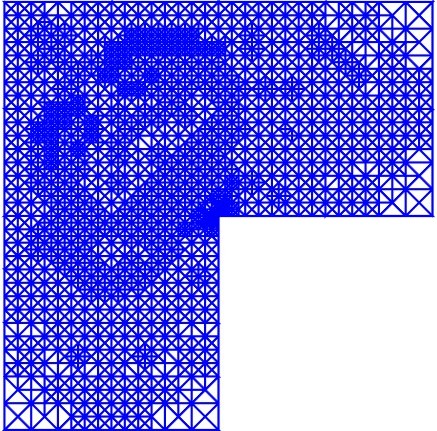}
	\end{minipage}
	\captionsetup{width=1\linewidth}
	\caption{Convergence history plot of $\lambda(1) - \mathrm{LEB}(1)$ (left) and adaptive triangulation with 4610 triangles (right) for the experiment in \Cref{sec:num_ex_eig}}
	\label{fig:biharm_eig}
\end{figure}

\subsection{Biharmonic eigenvalue problem}\label{sec:num_ex_eig}
In this subsection, we are interested in the computation of LEB for the biharmonic eigenvalue problem.
Recall the constants $\alpha = \sigma(1/\pi^4 + c_\mathrm{tr}/\pi^2 + c_\mathrm{tr} + c_\mathrm{tr}(2/\pi + n/\pi^2)) = 1.8355 \times \sigma$ and $\beta \coloneqq h^4/\pi^4$ from \Cref{thm:LEB}. To obtain LEB with higher-order convergence rates, $\sigma$ needs to be chosen such that $\alpha < 1$. In the benchmark below, we chose $\sigma = 0.4086$ ($\alpha = 3/4$).

\subsubsection{Adaptive algorithm}
Whenever $\alpha + \beta \lambda_h(j) > 1$, the mesh is refined uniformly.
Otherwise,
the refinement indicator
\begin{align*}
	\widehat{\eta}^2(T) \coloneqq \sum_{F \in \F(T)} (h_F^{-3} \|[R_h u_h]_F\|_F^2 + h_F^{-1} \|[\partial_n R_h u_h]_F\|_F^2)
	+ s_T(u_h,u_h)
\end{align*}
motivated from \Cref{thm:a-posteriori-eig}
is utilized in a standard adaptive loop with the D\"orfler marking strategy and bulk parameter $\vartheta = 1/2$ as above.

\subsubsection{Numerical benchmark}
We approximate the first eigenvalue $\lambda(1)$ of \eqref{def:pde-var} with the reference value $\lambda(1) = 418.9735$ from an adaptive computation with $k = 3$.
The convergence history of $\lambda(1) - \mathrm{LEB}(1)$ with $\mathrm{LEB}(1)$ from \Cref{thm:LEB} is displayed in \Cref{fig:biharm_eig} (left) with suboptimal convergence rates on uniformly refined meshes. The adaptive algorithm refines towards the expected singularity at the origin as displayed in \Cref{fig:biharm_eig} (right) and recovers the optimal convergence rates $k+1$ for all displayed polynomial degrees $k = 0, 1, 2$. We note that the preasymptotic regime observed in \Cref{fig:biharm_eig} (left)
arises from a possibly large overestimation of the constant in \Cref{lem:best-appr-stab-eig} leading to a small $\sigma$.
A remedy is the computation of accurate embedding constants as in \cite{Tran2024x2} using the hybrid high-order methodology of this paper.

\subsection{Conclusion}
In all examples, adaptive computation empirically recovers optimal convergence rates of the error quantities of interest. The analysis of this paper can be extended to more general (such as mixed or clamped) boundary conditions.

\appendix
\section{Unisolvence of the DOFs of $C^1$-element}\label{appendix}
In this section, we recall the degrees of freedom of $C^1$ conforming finite element in 2D \cite{Ciarlet2002} and 3D \cite{Zenisek1973,Zhang2009}. Since they are not suitable for the purpose of constructing the right-inverse $J_h$ of \Cref{sec:right-inverse} as point evaluations are involved, a modified set of degrees of freedom is provided below. Since the arguments to establish its unisolvence from \cite{Ciarlet2002,Zhang2009} carry over, we only provide an outline of the proofs.
Let $T$ denote a $n$ dimensional simplex with the vertices $x_j$ and associated nodal basis function $\Lambda_j$, $j = 1, \dots, n+1$.

\subsection{$C^1$- finite element in 2D}\label{appendix-1}
The $C^1$- finite element in 2D is known as the Argyris element (cf. \cite[Section 6]{Ciarlet2002}). For each $u\in P_k(T)$ with $k\geq 5$, the local degrees of freedom are given by:
\begin{enumerate}
	\item $u(x)$, $\nabla u(x)$, $\nabla^2u(x)$ at each vertex $x$ of $T$,
	\item $\int_e u q \ds$ for any $q\in P_{k-6}(e)$ and edge $e$ of $T$ (if $k=5$, this term vanishes),
	\item $\int_e \tfrac{\partial u}{\partial n_e} q \ds$ for any $q \in P_{k-5}(e)$ and edge $e$ of $T$ with unit normal vector $n_e$,
	\item $\int_T u q \dx$ for each $q\in P_{k-6}(T)$ (if $k=5$, this term vanishes).
\end{enumerate}
\begin{theorem}[$C^1$- finite element in 2D]\label{thm:C1_2D}
	Any $u\in P_k(T)$ with $k\geq 5$ is uniquely determined by the degrees of freedom from above.
\end{theorem}
\begin{proof}
	The number of degrees of freedom is
	\begin{align*}
		6\times 3 + (2k-9)\times 3 + \tfrac{(k-5)(k-4)}{2} = \tfrac{(k+1)(k+2)}{2}=\operatorname{dim}P_{k}(T)
	\end{align*}
	It suffices to prove that $u =0$ if it vanishes at all the degrees of freedom.
	From the first set of degrees of freedom, the derivatives of $u$ of order $\leq 2$ vanish at all the vertices, which indicates that $u|_e = (\lambda_j\lambda_k)^3p$ for some $p\in P_{k-6}(e)$ on an edge $e = \mathrm{conv}\{x_j,x_k\}$. Then the second set of degrees of freedom implies that $u|_e = 0$. Similarly, it holds that $\tfrac{\partial u}{\partial n_e}|_e = (\lambda_j\lambda_k)^2w$ for some $w\in P_{k-5}(e)$ and the third set of degrees of freedom provides $\tfrac{\partial u}{\partial n_e}|_e =0$. In summary, $u$ and $\nabla u$ vanishes on the boundary of $T$ and so, $u = (\lambda_1\lambda_2\lambda_3)^2v$ for some $v\in P_{k-6}(T)$. The fourth set of degrees of freedom concludes $u =0$.
\end{proof}
\subsection{$C^1$- finite element in 3D}\label{appendix-2}
For each $u\in P_k(T)$ with $k\geq 9$, the local degrees of freedom are given by:
\begin{enumerate}
	\item $D^\alpha(x)$ at each vertex $x$ of $T$ with $|\alpha|\leq 4$,
	\item $\int_e u q \dt$ for any $q\in P_{k-10}(e)$ and edge $e$ of $T$ (if $k=9$, this term vanishes),
	\item $\int_e \tfrac{\partial u}{\partial n_{e}^1}q \dt$, $\int_e\tfrac{\partial u}{\partial n_{e}^2}q \dt$ for any $q\in P_{k-9}(e)$ and edge $e$ of $T$ with two independent unit normal vectors $n_e^1,n_e^2$,
	\item $\int_e\tfrac{\partial^2 u}{\partial n_{e}^1\partial n_{e}^1}q \dt$, $\int_e\tfrac{\partial^2 u}{\partial n_{e}^1\partial n_{e}^2}q \dt$, $\int_e\tfrac{\partial^2 u}{\partial n_{e}^2\partial n_{e}^2} q \dt$ for any $q\in P_{k-8}(e)$ and each edge $e$ of $T$,
	\item $\int_f uq \ds$ for any $q\in P_{k-9}(f)$ and face $f$ of $\partial T$,
	\item $\int_f \tfrac{\partial u}{\partial n_f} q \ds$ for any $q\in P_{k-7}(f)$ and face $f$ of $T$ with unit normal vector $n_f$,
	\item $\int_Tuq \dx$ for each $q\in P_{k-8}(T)$.
\end{enumerate}
\begin{theorem}[$C^1$- finite element in 3D]\label{thm:C1_3D}
	Any $u\in P_k(T)$ with $k\geq 9$ is uniquely determined by the degrees of freedom from above.
\end{theorem}
\begin{proof}
	The number of degrees of freedom is
	\begin{align*}
		&35\times 4 + (6k-46)\times 6 + (\tfrac{(k-8)(k-7)}{2} +\tfrac{(k-6)(k-5)}{2})\times 4 + \tfrac{(k-7)(k-6)(k-5)}{6}\\
		= & \tfrac{(k+1)(k+2)(k+3)}{6} = \operatorname{dim}P_k(T)
	\end{align*}
	It suffices to prove that $u =0$ if it vanishes at all the degrees of freedom. From the first set of degrees of freedom, the derivatives of $u$ of order $\leq 4$ vanish at all the vertices, which indicates that $u|_e = (\lambda_j\lambda_k)^5p$ for some $p\in P_{k-10}(e)$ on an edge $e = \mathrm{conv}\{x_j,x_k\}$. Then the second set of degrees of freedom implies $u|_e = 0$. Similarly, the third and fourth set of degrees of freedom provide $\tfrac{\partial u}{\partial n_e^i}|_e = \tfrac{\partial^2u}{\partial n_e^i\partial n_e^j}=0$ for $i,j = 1,2$.
	This applies to all edges and,
	therefore, $u$, $\nabla u$, and $\nabla^2 u$, considered as a function on a two dimensional face $f = \mathrm{conv}\{x_j,x_k,x_\ell\}$, vanishes at the boundary of $f$. Hence,
	$u|_f = (\lambda_j\lambda_k\lambda_\ell)^3p$ for some $p\in P_{k-9}(f)$.
	The fifth set degrees of freedom provides $u|_f = 0$ and from the sixth set of degrees of freedom, we infer $\tfrac{\partial u}{\partial n_f}|_f=0$. Since $u$ and $\nabla u$ vanish on $\partial T$, $u = (\lambda_1\lambda_2\lambda_3\lambda_4)^2v$ for some $v\in P_{k-8}(T)$. The seventh set of degrees of freedom concludes $u = 0$.
\end{proof}
\printbibliography
\end{document}